\documentclass[11pt]{article}

\usepackage{hyperref}
\usepackage{inputenc}
\usepackage{amsmath}
\usepackage{amsfonts}
\usepackage{amssymb}
\usepackage{graphicx}
\usepackage[T1]{fontenc}
\usepackage[english]{babel}
\usepackage{mathrsfs}
\usepackage{amsthm}
\usepackage[all]{xy}
\usepackage{geometry}
\usepackage{enumerate}
\allowdisplaybreaks
\usepackage{color,xcolor}
\usepackage{dsfont}
\usepackage{appendix}
\title{Adaptation to a heterogeneous patchy environment with nonlocal dispersion}
\author{Alexis L\'{e}culier\thanks{Laboratoire Jacques-Louis Lions; UMR 7598; Sorbonne Universit\'e, Inria, CNRS, Universit\'e de Paris, , F-75005
Paris; E-mail : \texttt{leculier@ljll.math.upmc.fr}} \and Sepideh Mirrihami\thanks{Institut Montpelli\'{e}rain Alexander Grothendieck, Universit\'{e} Montpellier, CNRS, Montpellier, France; E-mail: \texttt{ sepideh.mirrahimi@umontpellier.fr} }}
\begin{document}                     
\maketitle

\definecolor{olivegreen}{RGB}{0, 100,0 }
\newcommand{\SM}{\textcolor{blue}}
\newcommand{\SMM}{\textcolor{blue}}
\newcommand{\AL}{\textcolor{red}}

\begin{abstract}
In this work, we provide an asymptotic analysis of the solutions to an elliptic integro-differential  equation. This equation describes the evolutionary equilibria of a phenotypically structured population, subject to selection, mutation, and both local and non-local dispersion in a spatially heterogeneous, possibly patchy, environment. Considering small effects of mutations, we provide an asymptotic description of the equilibria of the phenotypic density. This asymptotic description involves a Hamilton-Jacobi equation with constraint coupled with an eigenvalue problem. Based on such analysis, we characterize some qualitative properties of the phenotypic density at equilibrium depending on the heterogeneity of the environment. In particular, we show that when the heterogeneity of the environment is low,  the population concentrates around a single phenotypic trait leading to a unimodal phenotypic distribution. On the contrary, a strong fragmentation of the environment leads to multi-modal phenotypic distributions.

\end{abstract}

\newcommand{\e}{\varepsilon}
\newcommand{\R}{\mathbb{R}}

\newtheorem{theorem}{Theorem}
\newtheorem{corollary}{Corollary}
\newtheorem{lemma}{Lemma}
\newtheorem{definition}{Definition}
\newtheorem{proposition}{Proposition}
\newtheorem{example}{Example}
\newtheorem{definition-proposition}{Definition-Proposition}
\newtheorem*{notation}{Notation}
\newtheorem*{remark}{Remark}

\tableofcontents
\section{Introduction}
\subsection{The model and motivations}
We are interested in the study of  evolutionary equilibria of phenotypically structured populations in spatially heterogeneous and possibly patchy environments. Understanding the interplay between heterogeneous selection, migration and mutation is a major objective of evolutionary biology theory and could lead to a better understanding of the speciation process and the evolutionary response to global change  \cite{Whitlock2015}.
Joint evolution and spatial dispersion  {have} to be considered in the study of  species that need to adapt to climatic change, or in epidemiology, where pathogenic viruses or bacteria   may propagate in a population that has been partially vaccinated or  treated with antibiotics   \cite{HK.AL:2006,RH.JD.TH:12}. The mathematical problem then amounts to describing both phenotypic and spatial structure of the population, and connects to key questions in evolutionary ecology about the evolution of species ranges (where are the individuals?) and niches (which phenotypes are observed within the species?).  

We investigate  particularly  the effect of fragmented environment, considering non-local dispersion.   The non-local dispersion may have antagonistic effects on the population dynamics. On the one hand, it may allow the population to reach new favorable geographic regions which are not accessible by a local diffusion. On the other hand, it may also impede local adaptation by bringing individuals with locally maladapted traits from other regions. While the role of the nonlocal dispersion and the fragmentation of the environment is significant in many situations, as in the adaptation of forest trees to the climate change because of the effect of the wind on the seeds or the pollens, very few theoretical works take it into account \cite{ecology1}.

More precisely, we provide in this work an asymptotic analysis of the equilibria of a non-local parabolic Lotka-Volterra type equation, modeling the interplay between mutation, selection, local and non-local dispersion in possibly fragmented environments. The equation under study is the following   
\begin{equation} \label{E-0} 
\left\lbrace 
\begin{aligned}
& - \sigma_\theta \partial_{\theta \theta} n  -\sigma_{x}\partial_{xx} n  +\sigma_{x} L n  = n \left(R(x, \theta) - \kappa\rho (x) \right) &&\text{ in } \Omega \times ]-A, A[, \\
&\rho  (x) = \int_{]-A,A[} n (x,\theta) d \theta && \text{ in } \Omega, \\
& L n  (x, \theta) = \int_{\Omega} [n (x,\theta) - n (y,\theta)]K(x-y)dy && \text{ in } \Omega \times ]-A,A[,\\
&\partial_{\nu_x} n  (x, \theta )= 0 \ \text{ on } \partial \Omega \times ]-A,A[, \qquad \partial_{\nu_\theta} n  (x, \pm A) = 0   \ \text{ on } \Omega \times \left\lbrace \pm A \right\rbrace.
\end{aligned}
\right.
\end{equation}
with $\Omega$ a bounded subset of $\mathbb{R}$, representing the spatial domain. Here, $n (x,\theta)$ stands for the density of a population at equilibrium at position $x$ with a phenotypical trait $\theta$.  The term $R(x,\theta)$ corresponds to the intrinsic growth rate of individuals of phenotype $\theta$ at position $x$. The term $\rho(x)$ corresponds to the total size of the population at position $x$. Via the term $\kappa \rho$ in the right hand side of \eqref{E},  we take into account a mortality rate due to the uniform competition between the individuals at the same position,   with intensity $\kappa$. The trait of the parent is transmitted to the offspring. However the trait can be modified due to the mutations that we model by a Laplace term with respect to $\theta$. We also consider that the species is subject to a local and a non-local dispersion in the space variable $x$. Indeed, in addition to a classical local dispersion term modeled by a Laplace term with respect to $x$, we also take into account a non-local dispersion modeled by the integral operator $L$, assuming that the individuals can jump from position $x$ to  position $y$ with a rate $K(x-y)$. Finally, we have denoted by $\partial_{\nu_x}, \ \partial_{\nu_\theta}$ the exterior derivatives with respect to the variables $x$ and $\theta$.  The Neumann boundary condition with respect to $x$ models the fact that the species cannot leave the domain. The Neumann boundary condition with respect to $\theta$ means that the mutants cannot be born with a trait in $]-A,A[^c$. 

We are in particular interested in a regime where the mutations have small effects. To study such a situation, we set $\sigma_\theta=\e^2$, with $\e$ a small parameter. We also set $\sigma_x=\kappa=1$ to reduce the amount of notation. The equation on  the population density, denoted now by $n_\e$, is then written 
\begin{equation} \label{E} \tag{E}
\left\lbrace 
\begin{aligned}
&- \varepsilon^2 \partial_{\theta \theta} n_\varepsilon -\partial_{xx} n_\varepsilon  + L n_\varepsilon = n_\varepsilon(R(x, \theta) - \rho_\varepsilon(x)) &&\text{ in } \Omega \times ]-A, A[, \\
&\rho_\varepsilon (x) = \int_{]-A,A[} n_\varepsilon(x,\theta) d \theta && \text{ in } \Omega, \\
& L n_\varepsilon (x, \theta) = \int_{\Omega} [n_\varepsilon(x,\theta) - n_\varepsilon(y,\theta)]K(x-y)dy && \text{ in } \Omega \times ]-A,A[,\\
&\partial_{\nu_x} n_\varepsilon (x, \theta )= 0 \ \text{ on } \partial \Omega \times ]-A,A[, \qquad \partial_{\nu_\theta} n_\varepsilon (x, \pm A) = 0   \ \text{ on } \Omega \times \left\lbrace \pm A \right\rbrace.
\end{aligned}
\right.
\end{equation}
 Note   that when the mutation rate $\sigma_\theta$ is small compared to the dispersion rate $\sigma_x$,  equation \eqref{E-0} can be brought to the equation above via a change of variables. 
Our objective is to provide an asymptotic analysis of   $n_\e$ as the parameter $\e$ becomes vanishingly small.

Several questions motivate our analysis. Can we  determine extinction and survival criteria for this model? 
How would the population be spatially  distributed at equilibrium? Would all the spatial domain be  exploited close to its local carrying capacity or   would we observe formation of clusters in certain zones of the environment? How the population would be distributed phenotypically? Would the population be adapted locally everywhere, or  would we observe emergence of dominant traits? More specifically, would we observe emergence of generalist traits being adapted to an average environment, or specialist traits being adapted to certain zones of the environment? 
 What will be the impact of the the fragmentation of the habitat on the phenotypical distribution of the population? Would it lead to the emergence of specialist traits?

\subsection{The state of the art}

Related models  and the questions of spatial distribution and ecological niches were studied from a numerical point of view in the biological literature, considering local dispersion and a connected domain (see for instance \cite{MD.UD:03,JP.NB:05}). In particular, in \cite{JP.NB:05} formation of clusters was observed in a closely related model. While the authors suggested that the formation of such clusters would be a result of the bounded domain or small mutational effects, they did not provide any analytical support for such hypothesis. Such type of models, again in the context of local dispersion, was later derived from stochastic individual based models by Champagnat and M\'{e}l\'{e}ard in \cite{ChampagnatMeleard} where further numerical studies were provided (see also \cite{ADP} for a preliminary analysis of this model by Arnold, Desvilettes and Pr\'{e}vost). More recently, Alfaro, Coville and Raoul \cite{ACR} studied a closely related model, considering again local dispersion but in unbounded domains. They proved propagation phenomena and existence of traveling front solutions for parabolic equations close to \eqref{E} (see also the related works   \cite{BouinMirrahimi,ABR}). In the context of their study with unbounded domain, one expects indeed that the population would propagate and get adapted locally in every position in space attaining its local carrying capacity, which is in contrast with what we observe in bounded domains, in particular with what we obtain in the present work. However, in such model to our knowledge, yet there is no result characterizing rigorously the population distribution at the back of the front  (see however the work of Berestycki, Jin and Silvestre in \cite{BJS} in a particular case with spatially homogeneous growth rate).

A wide number of articles have also studied a  closely related equation known as the "cane-toads" model where the growth rate is independent of the trait, but the trait influences the ability of dispersal leading to a $\theta$ coefficient in front of the diffusion term in space (see for instance \cite{Turanova,BouinCalvez,BHR3,BHR2}).  This equation is motivated by the propagation of the cane toads in Australia by taking into account the role of a phenotypical trait: the size of the legs of the toads. More closely to our work, the steady states of a "cane-toads" type  model, in the regime of small mutations, were studied by Perthame and Souganidis in \cite{PerthameSouganidis} and by Lam and Lou in \cite{LamLou}. In another related project, a model where  similarly to \eqref{E} the growth rate, and not the dispersion rate, depends on the phenotype, but considering a discrete spatial structure, was studied by Mirrahimi and Gandon \cite{Mirrahimi17,MirrahimiGandon}. In these works an asymptotic analysis of the steady states in the regime of small mutations was provided. In particular, it was shown that the presence of spatial heterogeneity can lead to polymorphic situations, that is the emergence of several dominant traits in the population.

In this work, we will use an approach based on Hamilton-Jacobi equations, which is adapted to study the small mutation regime ($\varepsilon$ small). A closely related approach was first introduced in \cite{Opt_geom_2,Opt_geom_1}, by Friedlin using  probabilistic techniques and by Evans and Souganidis using deterministic tools, to study the propagation phenomena in reaction-diffusion equations. In the context of models from evolutionary
biology and in the regime of small mutations, this method was suggested by Dieckmann, Jabin, Mishler and Perthame \cite{DJMP}.  In \cite{BP}, Barles and Perthame provide the first rigorous results within this approach and obtain a concentration phenomena considering homogeneous environments:  as the mutational effects become small, the solution converges to a Dirac mass. In this case, the population at equilibrium is monomorphic (there is a single dominant trait in the population). We quote \cite{BMP} which extends the main results of \cite{BP}. This approach was then widely extended to study more general models with heterogeneity. In particular, in the context of the space heterogeneous environments, the works \cite{BouinMirrahimi,Turanova,PerthameSouganidis,Mirrahimi17} are within this framework. However, the analysis provided in these previous works do not allow to study  problem \eqref{E}.  The closest work is the one in \cite{BouinMirrahimi} which studies the propagation phenomenon in an unbounded domain, considering a different rescaling. Note also that none of the previous works considered a non-local dispersion operator, which adds significant difficulties to the analysis.

In an ecological context, fragmented environments and non-local spatial dispersion phenomena were studied by L\'eculier, Mirrahimi and Roquejoffre \cite{papier2} and by L\'eculier and Roquejoffre \cite{article3}. Both mentioned works do not take into account any phenotypical structure. In \cite{papier2}, the authors study invasion phenomena in a Fisher-KPP equation involving a fractional Laplacian arising in a fragmented periodic environment with Dirichlet exterior conditions. In \cite{article3}, the authors study the existence and uniqueness of bounded positive steady-states in a Fisher-KPP equation involving a fractional Laplacian in general fragmented environment with Dirichlet exterior conditions. One of the perspectives of the present work is to study models with other operators of dispersion, as the fractional Laplacian $(-\partial_{xx})^\alpha$, instead of $-\partial_{xx}+L$, and considering Dirichlet exterior conditions.

\subsection{The assumptions and the notations}

The domain $\Omega \subset \mathbb{R}$ is assumed to be bounded and composed of one or several connected components :
\begin{equation}\tag{H1} \label{H1Chap5}
\Omega = \underset{i=1}{\overset{m}{\bigcup}} \, ]a_i, b_i[ \qquad \text{ with } \quad a_1 < b_1 < a_2 <...<a_{m} <b_{m}.
\end{equation}
We assume that the growth rate verifies
\begin{equation}\tag{H2}\label{H2Chap5}
 R \in C^1\left\lbrace(\overline{\Omega }\times [-A,A] \right\rbrace) \quad \text{ and } \quad  \| R \|_{W^{1, \infty}( \Omega \times ]-A,A[)} < C_R.
\end{equation}
\begin{example}\label{example1}
A typical example of growth rate is written
\[R(x, \theta ) = r - g(bx -\theta)^2. \]
In this example, $r$ is the maximal growth rate. The above quadratic term indicates that the optimal trait at position $x$ is given by $\theta_o = bx$. The term $b$ is the gradient of the environment: it indicates how fast the optimal trait varies as a function of a position in space. Moreover, $g$ corresponds to the selection pressure. If $g$ increases, the habitats becomes more hostile for unsuitable individuals. 
\end{example}

We make the following assumptions on $K$
\begin{equation}\label{K}\tag{H3}
	K \in C^1(\Omega), \quad K > 0, \quad K(x) = K(-x),  \quad  0 < c_K < K < C_K \quad \text{ and }\quad  |\partial_x K| < C_K. 
\end{equation}

We introduce here two eigenvalue problems associated to the equation \eqref{E}: let $\lambda(\theta, \rho)$ be the principal eigenvalue of the operator $-\partial_{xx} - L - [R (\cdot, \theta) - \rho] Id$ and $\mu_\varepsilon$ be the principal eigenvalue of the operator $-\partial_{xx} -\varepsilon^2 \partial_{\theta \theta} - L - R$  with Neumann boundary conditions:
    \begin{equation}\label{vpchap5}
    \text{i.e.} \quad  \left\lbrace
    \begin{aligned}
    &-\partial_{xx}\psi^\theta+ L(\psi^\theta) - [R (\cdot, \theta) - \rho]\psi^\theta = \lambda(\theta, \rho)\psi^\theta && \text{ in } \Omega, \\
    & \partial_{\nu_x} \psi^\theta = 0 && \text{ in }  \partial  \Omega
    \end{aligned}
    \right.
    \end{equation}
    and 
    \begin{equation}\label{vpchap52}
\hspace{2.2cm} \left\lbrace
\begin{aligned}
&-\partial_{xx} \xi_\varepsilon - \varepsilon^2 \partial_{\theta\theta} \xi_\varepsilon + L \xi_\varepsilon  - R \xi_\varepsilon = \mu_\varepsilon \xi_\varepsilon  &&\text{ in } \Omega \times ]-A,A[, \\
& \partial_{\nu_x} \xi_\varepsilon = \partial_{\nu_\theta }\xi_\varepsilon = 0 && \text{ on } \partial ( \Omega \times ]-A,A[).
\end{aligned}
\right.
\end{equation}
All along the article, we consider that the principal eigenfunctions (such as $\psi^{\theta}$ or $\xi_\varepsilon$) are taken positive with $L^2$ norms equal to $1$. The existence and some properties of $\lambda$ are proved in Appendix A (the existence of $\mu$ follows similar arguments than the one of $\lambda$, therefore we let it to the reader).

We make the following assumption:
\begin{equation}\label{assumption}\tag{H4}
\exists \theta_0 \in ]-A,A[, \quad \text{ such that } \quad \underset{ \theta \in ]-A,A[}{\min} \lambda(\theta,0) =  \lambda(\theta_0, 0) < 0.
\end{equation}

\begin{lemma}\label{lemmamu}
Under the assumptions \eqref{H1Chap5}--\eqref{assumption}, we have
\[\mu_\varepsilon \underset{ \varepsilon \rightarrow 0 }{\longrightarrow } \lambda(\theta_0, 0).\]
\end{lemma}
It follows obviously that 
\begin{equation}\label{H5}
  \exists \varepsilon_0>0, \ \forall \varepsilon \in ]0, \varepsilon_0[, \quad   \mu_\varepsilon < \frac{\lambda(\theta_0, 0) }{2} < 0.
\end{equation} 
We postpone the proof of Lemma \ref{lemmamu} to the Appendix B and we make the hypothesis that \eqref{H5} holds true. 

\subsection{The results and the strategy}
\label{subsec:result}

First, we prove the following theorem which provides conditions for existence or non-existence of a solution of \eqref{E} for all small value $\varepsilon$.
\begin{theorem}\label{existence}
Under the assumptions \eqref{H1Chap5}--\eqref{assumption}, for all $\varepsilon \in ]0, \varepsilon_0[ $, there exists a non-trivial positive bounded solution $n_\varepsilon$ of \eqref{E}. If  Assumption \eqref{assumption} does not hold and $\lambda(\theta_0, 0)>0$, then there exists $\varepsilon_0>0$ small enough such that for all $\varepsilon < \varepsilon_0$, there does not exist a positive   solution $n_\varepsilon$ to \eqref{E}.
\end{theorem}

We expect indeed that in a dynamic version of \eqref{E}, the solution would converge in long time to a nontrivial stationary solution, that is a solution to \eqref{E}, when such non-trivial steady state exists and to $0$ otherwise. Admitting such property, the theorem above provides us with conditions of survival and extinction of the population. The survival condition means indeed that there exists at least one trait $\theta $ such that   $ \lambda(\theta, 0)<0$, so that such trait is viable in absence of competition.   

The proof of Theorem \ref{existence} follows the one of Theorem 2.1 by Lam and Lou in \cite{LamLou} which treats the case of local diffusion. This proof relies on a topological degree argument. In section \ref{sec:existence}, we provide the additional arguments which allow to adapt the proof of \cite{LamLou} to the non-local operator $L$.

Next, we perform the Hopf-Cole transformation 
\begin{equation}\label{H-C}
n_\varepsilon (x, \theta) = e^{\frac{u_\varepsilon (x, \theta)}{\varepsilon}}.
\end{equation}
This is the usual first step in the Hamilton-Jacobi approach (see \cite{DJMP, BP, BMP}). The main idea in this approach is to first study the limit of $u_\varepsilon $ as $\varepsilon \rightarrow 0$, and next obtain from this limit, information on the limit of the phenotypic density $n_\varepsilon$. The advantage of this transformation is that the limit of $u_\varepsilon$ is usually a continuous function which solves a Hamilton-Jacobi equation, while the limit of $n_\varepsilon$ is a measure. Performing such a change of variable, we find that $u_\varepsilon$ is solution to
\begin{equation} \label{EHCchap5} \tag{$E_{HC}$}
\left\lbrace 
\begin{aligned}
& -\frac{1}{\varepsilon} \partial_{xx} u_\varepsilon  - \frac{|\partial_x u_\varepsilon|^2}{\varepsilon^2}- \varepsilon \partial_{\theta \theta} u_\varepsilon  - |\partial_\theta u_\varepsilon |^2\\
& \   \hspace{2cm}    + \int_{\Omega} \left[ 1- e^\frac{u_\varepsilon (y) - u_\varepsilon(x)}{\varepsilon} \right] K(x-y) dy   = R(x, \theta) - \rho_\varepsilon(x) \quad &&\text{ in }  \   \Omega \times ]-A, A[, \\
&\rho_\varepsilon (x) = \int_{]-A,A[} n_\varepsilon(x,\theta) d \theta && \text{ in } \ ]-A,A[, \\
&\partial_{\nu_x} u_\varepsilon (x, \theta )= 0 \text{ on } \partial \Omega \times ]-A,A[, \quad \partial_{\nu_\theta} u_\varepsilon (x, \pm A) = 0  \text{ in } \Omega . 
\end{aligned}
\right.
\end{equation}
We prove the following
\begin{theorem}\label{main}
Under the assumptions \eqref{H1Chap5}--\eqref{assumption}, as $\varepsilon \rightarrow 0$ along subsequences, there holds
\begin{enumerate}
    \item $\rho_\varepsilon$ converges uniformly to a function $\rho \in C^1(\Omega)$ with 
    \[0<c \leq \rho \leq C,\]
    \item $u_\varepsilon$ converges uniformly to a continuous function $u$ with $u$ a viscosity solution of 
    \begin{equation}\label{HJChap5}
    \left\lbrace
    \begin{aligned}
    &- | \partial_\theta u(\theta) |^2=-\lambda(\theta, \rho) ,\\
    &\max \ u(\theta) = 0, \\
    & \partial_{\nu_\theta} u(\pm A)  = 0,
    \end{aligned}
    \right.
    \end{equation}
    where $\lambda(\theta, \rho)$ is the principal eigenvalue introduced in \eqref{vpchap5}.
    Moreover, the limit $u$ depends only on $\theta$.  
    \item $n_\varepsilon$ converges to a measure $n$ in the sense of measures. Moreover, 
    \begin{equation}
    \label{inclusion}
    \mathrm{supp } \  n \subset \Omega\times \left\lbrace u(\theta) = 0 \right\rbrace \subset \Omega \times \left\lbrace \lambda(\theta,\rho)=0\right\rbrace.
    \end{equation}
\end{enumerate}
\end{theorem}
The theorem above allows us to characterize the phenotypic density $n$, at the limit as $\varepsilon\to 0$, via the Hamilton-Jacobi equation with constraint \eqref{HJChap5} coupled with the eigenvalue problem \eqref{vpchap5} and the inclusion properties \eqref{inclusion}. We expect indeed that $n$ would be a sum of Dirac masses in $\theta$ as follows:
$$
n(x,\theta)=\sum_{i=1}^d \rho_i(x) \delta(\theta-\theta_i).
$$
In section \ref{sec:qual}, we will use the information obtained above to characterize the phenotypic density $n$ in some particular situations. On the one hand, we will identify a situation where the phenotypic density at the limit will be a single Dirac mass in $\theta$ corresponding to a monomorphic population. On the other hand, we will show that a strong fragmentation of the environment will   lead to polymorphic situations.
 
 \bigskip

Let us present briefly below the main ingredients to prove \eqref{HJChap5}.   We first provide  heuristic arguments to understand how the Hamilton-Jacobi equation in \eqref{HJChap5} can be recovered. To this end, we perform asymptotic developments of $u_\varepsilon$ and $\rho_\varepsilon$ with respect to the powers of $\varepsilon$
\[ \text{ i.e. }\quad u_\varepsilon (x, \theta) = u_0(x, \theta) + \varepsilon u_1(x, \theta)  +  o(\varepsilon) \quad \text{ and } \quad \rho_\varepsilon (x) = \rho_0(x) + o_\varepsilon(1).\]
Next, we implement such asymptotic developments into \eqref{EHCchap5}, to obtain
\begin{align*}
    &\frac{1}{\varepsilon} \left( -\partial_{xx} u_0 -2|\partial_{x}u_0 \partial_x u_1| -  \frac{|\partial_{x} u_0|^2}{\varepsilon} \right) + \int_\Omega [1-e^{\frac{u_0(y, \theta) - u_0(x, \theta)}{\varepsilon} + u_1(y, \theta) - u_1(x, \theta) + o_\varepsilon(1)}]K(x-y)dy \\
    &-\partial_{xx} u_1 - |\partial_x u_1|^2 - |\partial_\theta u_0|^2 - [R - \rho_0] + o_\varepsilon(1) = 0.
\end{align*}
We then organize the equation by powers of $\e$. Keeping the terms of order $\e^{-2}$, we find
\[ \partial_x u_0(x,\theta)=0 \quad \Rightarrow  \quad u_0(x, \theta) = u_0(\theta).\]
Moreover, keeping the terms of order $\e^0$, we deduce that
\begin{equation}\label{IdeeFormelle2}
-|\partial_\theta u_0 (\theta)  |^2 = [R(x, \theta) - \rho_0(x) ] + \partial_{xx} u_1(x, \theta) + |\partial_x u_1(x,\theta)|^2 - \int_{\Omega}[1-e^{u_1(y, \theta) - u_1(x, \theta) }]K(x-y)dy.
\end{equation}
Setting $\widetilde \psi=\exp(u_1)$ and replacing in the equation above we obtain that
$$
- \partial_{xx} \widetilde \psi + L(\widetilde \psi)-  [R(x, \theta) - \rho_0(x) ] \widetilde \psi= |\partial_\theta u_0 (\theta)  |^2 \widetilde \psi.
$$
Since $-|\partial_\theta u_0 (\theta)  |^2$ does not depend   on $x$ and $\widetilde \psi$ is a positive function, the equation above suggests that $\widetilde \psi$ and $-|\partial_\theta u_0 (\theta)  |^2 $ are respectively the   principal eigenfunction and eigenvalue introduced in \eqref{vpchap5}. This leads in particular   to the Hamilton-Jacobi equation in \eqref{HJChap5}.

From a technical point of view, the convergence of $(u_\varepsilon)_{\varepsilon > 0}$ is proved using the Arzela-Ascoli Theorem and a perturbed test function argument (see \cite{Evans_visco_1}).   In order to apply the Arzela-Ascoli Theorem, we provide first some regularity estimates on $u_\e$. We prove in particular, using the Bernstein's method, that the first derivatives are bounded. These bounds rely on the establishment of Harnack type inequalities. 
More precisely, we prove the following regularity results on $u_\varepsilon$.
\begin{theorem}\label{Regularitythm}
Under assumptions \eqref{H1Chap5}--\eqref{assumption}, the following results hold true.
\begin{enumerate}
\item $\mathrm{[Harnack \  inequality]}  \quad $  There exists a constant $C>0$ (independent of the choice of $\varepsilon$) such that for all intervals $I \subset ]-A,A[$ with $|I| = \varepsilon$, there holds
\begin{equation}
\label{Harnack}
    \underset{(x,\theta) \in \Omega \times I}{\sup} n_\varepsilon (x, \theta) \leq C \ \underset{(x,\theta) \in \Omega \times I}{\inf} n_\varepsilon (x,\theta).
\end{equation}
\item $\mathrm{[Lipschitz \ bounds]} \quad $ There exists $C>0$ such that for all $\varepsilon$ small enough,
\begin{equation}\label{Lipeq}
    |\partial_x u_\varepsilon | \leq C \varepsilon \quad \text{ and } \quad |\partial_\theta u_\varepsilon | < C.
\end{equation}
\item $\mathrm{[Bounds \ on \ \rho_\varepsilon]} \quad $For all $\varepsilon$ small enough, $\rho_\varepsilon$ is uniformly bounded in $W^{2,p}(\Omega)$ for all $p \in [1, +\infty]$. Moreover, there exists $c,C>0$ (independent of the choice of $\varepsilon$) such that 
\begin{equation}\label{rhoin}
c \leq \rho_\varepsilon \leq C.
\end{equation}
\item $\mathrm{[Bounds \ on \ u_\varepsilon]} \quad $The following holds true
\begin{equation}\label{boundueps}
\underset{ \varepsilon \rightarrow 0}{\lim} \  \underset{(x,\theta) \in \Omega \times ]-A,A[}{\sup} u_\varepsilon \leq 0 \quad \text{ and } \quad -a < \underset{ \varepsilon \rightarrow 0}{\lim} \  \underset{(x,\theta) \in \Omega \times ]-A,A[}{\inf}  u_\varepsilon
, \end{equation}
with $a>0$.
\end{enumerate}
\end{theorem}
\begin{remark}
In item \textit{1.} of Theorem \ref{Regularitythm}, the interval $I$ can be at the boundary of $]-A,A[$
\[\text{i.e.} \quad I = ]-A, -A + \varepsilon[ \quad \text{ or } \quad I = ]A-\varepsilon, A[.\]
\end{remark}

The combination of the local and the non-local diffusion terms makes the establishment of such regularity estimates non-standard (see for instance \cite{BMP} and \cite{Mirrahimi12} where such types of estimates were obtained for related models with a local diffusion term). Here, the Harnack inequality \eqref{Harnack} is used to obtain the Lipschitz regularity estimate \eqref{Lipeq}. However, the result is by itself interesting since it extends the classical Harnack inequality to elliptic operators with joint local and nonlocal diffusion terms. 

Note finally that the constraint $\max_\theta u(\theta)=0$ is a consequence of the Hopf-Cole transformation \eqref{H-C} and the fact that $\rho_\e$ remains bounded   away from $0$, uniformly in $\e$. For a detailed proof of Theorem \ref{main} we refer to Section \ref{convergence}.

\subsection{Outline of the paper}
 In section \ref{sec:qual}, we focus on the qualitative properties of the phenotypic density $n$ and show which types of qualitative conclusions we can obtain thanks to our theoretical results presented above.  In section \ref{discussion}, we investigate, by some numerical simulations,  the qualitative properties of $n$ and confirm our theoretical results.
In section \ref{sec:existence}, we provide the existence of $n_\varepsilon$ by proving Theorem \ref{existence}. Next, we prove the regularity results given by Theorem \ref{Regularitythm} in section \ref{regularity}. Section \ref{convergence} is devoted to the proof of Theorem \ref{main}. \\
 As mentioned above, Appendix A is devoted to the existence of the eigenvalues $\lambda(\theta, \rho)$. We also prove some results in Appendix A that are stated and used in section \ref{sec:qual}. Lemma \ref{lemmamu} is proved in Appendix B.  \\
The constants $c,C$ are positive constants independent of the choice of $\varepsilon$ and may change from line to line when there is no confusion possible.

\section{Qualitative properties of the population density $n$ }
\label{sec:qual}

The objective of this section is to show how the asymptotic results provided in Theorem \ref{main} imply qualitative results on the population density $n$, at the limit as $\e\to 0$. Recall from Theorem \ref{main} that as $\e\to 0$ and along subsequences, $n_\e$ tends weakly to a measure $n$ and $\rho_\e$ converges uniformly to a function $\rho$ such that
$$
\mathrm{supp}\, n \subset \Omega \times \{\theta\ |\ u(\theta)=0\}\subset \Omega\times \{\theta\ | \lambda(\theta,\rho)=\min_{\theta'} \lambda(\theta',\rho) = 0\},
$$
with $u:[-A,A]\to \R$  the viscosity to 
\[
    \left\lbrace
    \begin{aligned}
    &- | \partial_\theta u(\theta) |^2=-\lambda(\theta, \rho) \quad \theta \in ]-A,A [,\\
    &\partial_{\nu_\theta} u(\pm A)=0, \\
    &\max \ u(\theta) = 0,
    \end{aligned}
    \right.
\]
and $\lambda(\theta,\rho)$ the principal eigenvalue corresponding to the following problem:
\[
   \left\lbrace
    \begin{aligned}
    &-\partial_{xx}\psi^\theta+ L(\psi^\theta) - [R (\cdot, \theta) - \rho]\psi^\theta = \lambda(\theta, \rho)\psi^\theta && \text{ in } \Omega, \\
    & \partial_{\nu_x} \psi^\theta = 0 && \text{ in }  \partial  \Omega.
    \end{aligned}
    \right.
\]
Moreover, from the first item of Theorem \ref{Regularitythm} we deduce  that 
$$
\text{if $\theta_0 \in \mathrm{supp}\, n(x_0,\cdot)$ for some $x_0\in \Omega$,\qquad then $\theta_0 \in \mathrm{supp}\, n(x,\cdot)$ for all $x\in \Omega$.}
$$
Consequently, we have
$$
\mathrm{supp}\, n =\Omega\times \Gamma_\theta,
\qquad
\Gamma_\theta\subset \{\theta | \, u(\theta) = 0\} \subset \{\theta\ | \lambda(\theta,\rho)=\min_{\theta'} \lambda(\theta',\rho) = 0\}.
$$
Note  that it may happen that $u(\theta) = 0$ for some $\theta \in [-A,A]$ but that $\theta$ does not belong to $\Gamma_\theta$. Note also that Theorem \ref{main} guarantees convergence of $(n_\e,\rho_\e)$ to  $(n,\rho)$, only along subsequences. It does not exclude possibility of multiple limits $(n,\rho)$ as $\e\to 0$. 
\\

We expect indeed that $u$ would take its maximum at some distinct traits such that the phenotypic density $n$ would have the following form:
$$
n(x,\theta)=\sum_{i=1}^d \rho_i(x) \delta(\theta-\theta_i),\qquad \rho(x)=\sum_{i=1}^d \rho_i(x) , \qquad \rho_i(x)>0.
$$
This expectation motivates the following definitions.
\begin{definition}
\begin{itemize}
\item Any trait $\theta\in \Gamma_\theta$ is called an \texttt{emergent trait}. 
\item A population density is called \texttt{monomorphic} if the set of emergent traits, that is $\Gamma_\theta$,  is reduced to a single point.
\item A population density is called \texttt{polymorphic} if it is not monomorphic.  \\
\end{itemize}
\end{definition}
With these definitions, any monomorphic population density is a  Dirac mass with respect to $\theta$,
\[i.e. \qquad n(x,\theta)=\rho(x)\delta(\theta-\overline{ \theta}).\]

In subsection \ref{sec:unique-mono} we will show that, for a particular choice of $R$, there is at most one possible monomorphic limit $(n,\rho)$. In particular under symmetry conditions on the set $\Omega$, the only possible monomorphic outcome at the limit would be $n(x,\theta)=\rho(x)\delta(\theta)$. We next identify a situation in subsection \ref{sec:mono} where the phenotypic density $n$ is indeed monomorphic. Finally we show in subsection \ref{sec:poly} that a strong fragmentation of the environment may lead to polymorphic situations. 

\bigskip

Before providing our qualitative results, we state the following technical result on the principal eigenvalue $\lambda$ that will help us in the next subsections to obtain our qualitative results. This result is proved in Appendix A. 
\begin{proposition}\label{derivlambda}
Under the assumptions \eqref{H1Chap5}--\eqref{K}, the following identity holds true:
\begin{equation}\label{derivvp}
    \partial_{\theta} \lambda(\theta, \rho )  =  - \int_\Omega \partial_{\theta } R(x, \theta) \psi^\theta (x)^2 dx.
\end{equation} 
\end{proposition}

\begin{corollary}\label{coroderiv}
Assume \eqref{H1Chap5}--\eqref{assumption} and let $\overline{\theta} \in \Gamma_\theta$ be an emergent trait. We have 
\[ \int_\Omega \partial_{\theta } R(x, \overline{\theta}) \psi^{\overline{\theta}}(x)^2 dx = 0.\]
\end{corollary}

Next, we consider some examples with explicit expressions of $R$. We illustrate how Proposition \ref{derivlambda} can be useful to characterize the emergent traits. 

\vspace{0.5cm}

\textbf{Example A. } We fix $\theta_0 \in ]-A,A[$, and we define 
\[ R(x, \theta) = r - g(\theta - \theta_0)^2.\]
We assume that $r$ is large enough such that \eqref{assumption} holds true for $\theta = \theta_0$. From Corollary \ref{coroderiv}, we deduce that at any emergent trait $\overline{\theta}$, we have
\[ 2g(\overline{\theta} - \theta_0)  \int_{\Omega } \psi^{\overline{\theta}}(x)^2 dx = 0.\]
We deduce that the unique emergent trait is $\overline{\theta} = \theta_0$. Therefore, the limit population is monomorphic. Of course, this example is a toy-model and does not involve any spatial structure. 

\vspace{0.1cm}

\textbf{Example B. } We define 
\begin{equation}
\label{fitness-quadratic}
 R(x, \theta) = r - g(\theta - bx)^2.
\end{equation}
From Corollary \ref{coroderiv}, we deduce that for any emergent trait $\overline{\theta}$, we have 
\[ \overline{\theta} = b \int_\Omega x \psi^{\overline{\theta}}(x)^2 dx.\]
This is not enough to conclude that the number of emergent traits is finite. However, we can still remark that the emergent traits  are fixed points of the application 
\begin{equation}\label{chi}
    \chi : \theta \mapsto b \int_{\Omega } x \psi^\theta(x)^2  dx.
\end{equation}

\subsection{At most one possible monomorphic outcome}
\label{sec:unique-mono}

In this subsection,  we restrict our study to monomorphic limits. We will prove that, when $R$ is given by \eqref{fitness-quadratic}, the problem admits at most one monomorphic limit.

\begin{proposition}
\label{prop:mono-uniq}Assume \eqref{H1Chap5}--\eqref{assumption} and let $n_1(x,\theta)=\rho_1(x)\delta(\theta-\overline \theta_1)$ and $n_2(x,\theta)=\rho_2(x)\delta(\theta-\overline \theta_2)$ be two monomorphic limits of the problem. Then, $\rho_1=\rho_2$.\\
 Additionally, if $R$ is given by \eqref{fitness-quadratic}, then $\theta_1=\theta_2$.\\
 Furthermore, if the domain $\Omega$ is symmetric with respect to $x=0$, then the only possible emergent trait corresponding to a monomorphic population is $\overline{\theta}=0$. 
\end{proposition}

To prove  Proposition \ref{prop:mono-uniq}, we will use the following Rayleigh quotient: 
\begin{equation}\label{RayleighQuotient}
    \mathcal{R}(\theta, \rho, \phi) =  \frac{  \int_\Omega |\partial_x \phi|^2 dx + \frac{ \int_{\Omega \times \Omega} [\phi(x) - \phi(y)]^2 K(x-y) dxdy}{2} - \int_\Omega [R(x, \theta) - \rho(x)] \phi^2 dx }{\int_{\Omega} \phi(x)^2dx},
\end{equation}
We recall the classical link between $\lambda$ and $\mathcal{R}$:
\[ \lambda(\theta, \rho) = \underset{\phi \in H^1(\Omega)}{\min} \mathcal{R}(\theta, \rho, \phi). \]
We will also use the following lemma.
\begin{lemma}\label{lem:rhoef}
Assume \eqref{H1Chap5}--\eqref{assumption}and that $(n_\e)$ converges in $\mathrm{L^\infty}\left(\mathrm{w*} (0,\infty) ; \mathcal{M}^1(\R^d) \right)$ to $\rho(x)\delta(\theta-\theta_0)$. Then $\rho(\cdot)$ is the principal eigenfunction corresponding to the operator $-\Delta_x + L -[R(\cdot, \theta_0) -\rho ]$.
 \end{lemma}
 \begin{proof}
 By passing weakly to the limit in the equation \eqref{E}, we obtain that
 $$
 -\partial_{xx}\rho+L\rho=\rho(R(x,\theta_0)-\rho),
 $$
 which implies that $\rho$ is the principal eigenfunction corresponding to the operator $-\Delta_x + L -[R(\cdot, \theta_0) -\rho ]$.
 \end{proof}

\begin{proof}[Proof of Proposition \ref{prop:mono-uniq}]
 Let $n_1(x,\theta)=\rho_1(x)\delta (\theta-\theta_1)$ and $n_{2}(x,\theta)=\rho_{2}(x)\delta (\theta-\theta_{2})$ be two possible limits. Our objective is to prove that
\[ \begin{aligned}
 &\rho_1=\rho_2 \quad &&\text{ for generic } R \text{ satisfying \eqref{H2Chap5}}, \\
 \text{and } \ &  \theta_1 = \theta_2 &&\text{ if, furthermore,  } R \text{ is of the form of \eqref{fitness-quadratic}}.
 \end{aligned}\]
From Theorem \ref{main} we obtain that
$$
\min_\theta \lambda(\theta,\rho_i)=\lambda(\theta_i,\rho_i)=0.
$$
Let $\psi_{i}$ be the positive eigenfunction  associated to the operator $-\Delta_x + L -[R(\cdot, \theta_i) -\rho_i]$ 
 by Proposition \ref{propev} with $\|\psi_i \|_{L^2} = 1$ (i.e $\psi_i = \psi^{\theta_i}$).  Notice that since the limit is monomorphic and thanks to Lemma \ref{lem:rhoef}, we have for each case $\rho_i = c_i \psi_i$ with $c_i>0$.  
Using the Rayleigh quotient introduced in \eqref{RayleighQuotient}, we deduce that 
\begin{equation} \label{Rtheta1psi1psi1}
\begin{aligned}
   0&= \int_{\Omega} c_1^2 \psi_1(x)^2dx \times \mathcal{R}(\theta_1,c_1 \psi_1,c_1\psi_1) \\
   & = c_1^2 \int_\Omega |\partial_x \psi_1|^2 dx +c_1^2 \frac{ \int_{\Omega \times \Omega} [\psi_1(x) - \psi_1(y)]^2 K(x-y) dxdy}{2} - c_1^2 \int_\Omega [R(x, \theta_1) -c_1 \psi_1] \psi_1^2 dx, 
\end{aligned}
\end{equation}
and 
\begin{equation}    \label{Rtheta1psi2psi1}
\begin{aligned}
    &0 && \leq \int_{\Omega} c_1^2 \psi_1(x)^2dx \times\mathcal{R}(\theta_1,c_2\psi_2, c_1 \psi_1)\\
    & \ && =  c_1^2  \int_\Omega |\partial_x \psi_1|^2 dx + c_1^2 \frac{ \int_{\Omega \times \Omega} [\psi_1(x) - \psi_1(y)]^2 K(x-y) dxdy}{2} - c_1^2 \int_\Omega [R(x, \theta_1) - c_2 \psi_2] \psi_1^2 dx. 
 \end{aligned}
\end{equation}

The last inequality holds since thanks to Theorem \ref{main} we have
\[ 0 \leq \lambda(\theta_1, c_2\psi_2) \leq  \mathcal{R}(\theta_1,c_2\psi_2, c_1 \psi_1).\]
By substracting \eqref{Rtheta1psi1psi1} to \eqref{Rtheta1psi2psi1}, we deduce that 
\begin{equation}
    \label{psi12}
     \int_{\Omega} \left(  c_1 \psi_1(x)\right)^{3} dx \leq  \int_{\Omega}  \left(c_1\psi_1(x)\right)^2 c_2\psi_2(x) dx.
\end{equation}
Following similar computations, we obtain that 
\begin{equation}
    \label{psi21}
    \int_{\Omega} \left(c_2\psi_2(x) \right)^3 dx \leq   \int_{\Omega} \left( c_2\psi_2(x) \right)^2 c_1\psi_1(x) dx .
\end{equation}
By combining \eqref{psi12} and \eqref{psi21}, we deduce that
\[ \int_{\Omega} [c_1\psi_1 (x) - c_2\psi_2(x) ]^2 (c_1\psi_1(x) + c_2\psi_2(x) ) dx \leq 0. \]
Since $c_{1,2}\psi_{1,2} > 0$, we deduce that $c_1\psi_1 = c_2\psi_2$, and hence $\rho_1=\rho_2$. \\
We next assume that $R$ is given by \eqref{fitness-quadratic} and conclude thanks to Corollary \ref{coroderiv}. Indeed, since $\partial_\theta R(x, \theta_i) = -2g(\theta_i - bx)$, we deduce that 
\[ \theta_1 = \frac{-b}{2g}\int_{\Omega} x c_1\psi_1^2(x) dx  =  \frac{-b}{2g}\int_{\Omega} x c_2\psi_2^2(x) dx  = \theta_2.\]
 
 Finally, note that under symmetry conditions, if $n(x,\theta)=\rho(x)\delta(\theta)$ is a possible monomorphic limit, then $\widetilde n(x,\theta)=n(-x,-\theta)$ is also a possible monomorphic outcome. We hence deduce in this case that $\theta=0$.
\end{proof}

\subsection{A small selection pressure leads to a monomorphic population}
\label{sec:mono}

First, we state a technical result about the dependance of $\lambda(\theta, \rho)$ with respect to $g$. This proposition is proved at the end of Appendix A. 
\begin{proposition}\label{propo:lambda:result2}
Assume that $R$ is given by \eqref{fitness-quadratic}. Under the hypothesis \eqref{H1Chap5}--\eqref{K}, there holds:
\begin{equation}\label{Claim1}
\lambda(\theta, \rho) \text{ is non-decreasing with respect to }g,
 \end{equation}
and 
\begin{equation}\label{Claim2}
 \int_{\Omega} |\partial_\theta \psi^\theta(x)|^2dx \underset{g \rightarrow 0}{\longrightarrow }0.
 \end{equation}
Moreover, the above limit is uniform with respect to $\theta$. 
\end{proposition}

Next, we use Proposition \ref{propo:lambda:result2} to provide a condition which ensures the existence of a set of parameters such that the limit is monomorphic.
\begin{proposition}\label{propofinite}
Assume that $R$ is given by \eqref{fitness-quadratic}. Under the hypothesis \eqref{H1Chap5}--\eqref{assumption}, there exists $g_0>0$ such that if $g\in ]0, g_0[$ then there exists a unique emergent trait. 
\end{proposition}

\begin{proof}
First, we differentiate $\chi$ (defined by \eqref{chi}) with respect to $\theta$, to obtain that 
\[ \chi'(\theta)  = 2b\int_{\Omega} x \psi^\theta (x) \partial_{\theta } \psi^\theta(x)dx.\]
Thanks to the Cauchy-Schwarz inequality and Proposition \ref{propo:lambda:result2}, we deduce that
\[ |\chi'(\theta) | \leq 2b\, \underset{ x \in \Omega}{\sup}\  |x| \ \int_{\Omega}  \psi^\theta (x)^2 dx  \int_{\Omega}   \partial_{\theta } \psi^\theta(x)^2 dx =  2b \ \underset{ x \in \Omega}{\sup} \  |x| \ \int_{\Omega}   \partial_{\theta } \psi^\theta(x)^2 dx \underset{ g \rightarrow 0}{\longrightarrow } 0 .\]
Since the last inequality does not depend on the choice of $\theta$, we deduce the existence of a uniform $g_0>0$ such that for all $g \in ]0, g_0[$ we have 
\[ |\chi'(\theta) | < 1.\]

Thanks to Theorem \ref{main} there exists at least one fixed point $\overline{\theta}$ to $\chi$. Moreover since $\chi$ is a contraction mapping, we recover that the fixed point $\overline{\theta}$ is unique. 
\end{proof}
 
\subsection{A strong fragmentation of the environment  leads to polymorphism}
\label{sec:poly}

In this subsection we consider the growth rate given by \eqref{fitness-quadratic} and spatial domains of the following type:
$$
\Omega_d=(-d-a,-d)\cup (d,d+a).
$$

\begin{proposition}\label{prop:polydist}
Under the assumptions \eqref{H1Chap5}--\eqref{assumption}, for $d\geq d_0$, with $d_0$ a large enough constant, the trait $0$ is not included in the support of the phenotypic density $n$. As a consequence, the population density is not monomorphic. 
\end{proposition}

\begin{proof}
We first note that, for fixed $d$, there always exists $r(d)$ such that for all $r\geq r(d)$, \eqref{assumption} is satisfied so that thanks to Theorem \ref{existence} the population persists. We thus can assume that, up to adjusting the constant $r$, we are in a situation where the population persists.

We prove that $\theta=0$ is not an emergent trait. Let's suppose by contradiction that $\theta = 0$ is included in the support of $n(x,\cdot)$. 
Then, the Rayleigh quotient \eqref{RayleighQuotient} implies that
$$
\lambda(0,\rho)=\inf_{{\phi\in H^1(\Omega)}} \mathcal R (0, \rho,\phi)=0\leq \lambda(\theta,\rho), \quad \text{for all $\theta\in [-A,A]$}.
$$
Let $\phi_0\in H^1(\Omega)$ be such that
$$
\inf_{{\phi\in H^1(\Omega)}} \mathcal R (0, \rho,\phi)=\mathcal R (0, \rho,\phi_0)=0, \qquad  \|\phi_0\|_{L^2(\Omega)}=1, \qquad \phi_0>0.
$$
We choose 
$$
\theta_1=-d-a/2, \qquad \theta_2=d+a/2.
$$
We also define
$$
\Omega_1=[-d-a,-d], \quad \Omega_2=[d,d+a], \quad \phi_0^{(1)}=\phi_0 \mathds{1}_{\Omega_1}, \qquad \phi_0^{(2)}=\phi_0 \mathds{1}_{\Omega_2}.
$$
We will prove that when $d$ is large enough, 
\begin{equation}
\label{lambdatheta0}
\| \phi_0^{(1)}\|_{L^2}\lambda(\theta_1,\rho)+\| \phi_0^{(2)}\|_{L^2}\lambda(\theta_2,\rho)< \lambda(0,\rho)=0.
\end{equation}
Since $\| \phi_0^{(1)}\|_{L^2} + \| \phi_0^{(2)}\|_{L^2} = 1$, the inequality \eqref{lambdatheta0} would imply 
\[\min(\lambda(\theta_1, \rho), \lambda(\theta_2, \rho) )< 0\]
which is in contradiction with the positiveness of the eigenvalues $\lambda(\theta, \rho)$ established in \textit{2.} of Theorem \ref{main}.

We consider the positive function $\phi_i$, for $i=1,2$, which minimizes $\mathcal R( 0,\rho,\cdot)$ restricted to the set $\Omega_i$, that is the following operator
$$
\mathcal R( 0,\rho,\Omega_i,\phi)=\frac{\int_{\Omega_i}|\partial_x \phi|^2 dx + \frac{ \int_{\Omega_i \times \Omega_i} [\phi(x) - \phi(y)]^2 K(x-y) dxdy}{2} - \int_{\Omega_i} [R(x,0) - \rho(x)] \phi^2 dx}{\|\phi\|_{L^2(\Omega_i)}^2},
$$
and such that
$$
\|\phi_i\|_{L^2(\Omega_i)}=1.
$$
Note that here $\phi_i$ has its support in $\Omega_i$, while $\phi_0$ has its support in $\Omega_d$.
To prove \eqref{lambdatheta0}, it  is enough to prove that 
$$
\| \phi_0^{(1)}\|_{L^2}^2\mathcal R(\theta_1,\rho,\phi_1) + \| \phi_0^{(2)}\|_{L^2}^2\mathcal R(\theta_2,\rho,\phi_2)<\mathcal R(0,\rho,\phi_0).
$$
We compute
$$
\begin{array}{rl}
\| \phi_0^{(i)}\|_{L^2}^2\mathcal R(\theta_i,\rho,\phi_i) &=\| \phi_0^{(i)}\|_{L^2}^2 \big[\mathcal R( \theta_i,\rho,\Omega_i,\phi_i)+\frac{1}{2} \int_{\underset{i\neq j}{\Omega_i \times \Omega_j}} \phi_i^2(x)  K(x-y) dxdy \big] \\
&\leq \| \phi_0^{(i)}\|_{L^2}^2 \big[\mathcal R( \theta_i,\rho,\Omega_i,\phi_0^{(i)})+\frac{1}{2}  \int_{\underset{i\neq j}{\Omega_i \times \Omega_j}} \phi_i^2(x)  K(x-y) dxdy\big] \\
&=\int_{\Omega_i}|\partial_x \phi_0^{(i)}|^2 dx + \frac{1}{2} \int_{\Omega_i \times \Omega_i} [\phi_0^{(i)}(x) - \phi_0^{(i)}(y)]^2 K(x-y) dxdy\\\\
& - \int_{\Omega_i} [R(x,\theta_i) - \rho(x)] \phi_0^{(i)\, 2}(x) dx
+\frac{1}{2}\| \phi_0^{(i)}\|_{L^2}^2  \int_{\underset{i\neq j}{\Omega_i \times \Omega_j}} \phi_i^2(x)  K(x-y) dxdy  .
  \end{array}
$$
Combining the inequality above, for $i=1$ and $i=2$, we obtain that
$$
\begin{array}{c}
\| \phi_0^{(1)}\|_{L^2}^2\mathcal R(\theta_1,\rho,\phi_1) + \| \phi_0^{(2)}\|_{L^2}^2\mathcal R(\theta_2,\rho,\phi_2)\leq\mathcal R(0,\rho,\phi_0)
\\
+\int_{\Omega_1} [R(x,0)-R(x,\theta_1)  ] \phi_0^{2}(x) dx+\int_{\Omega_2} [R(x,0)-R(x,\theta_2)  ] \phi_0^{2}(x) dx\\
+\frac{1}{2}\| \phi_0^{(1)}\|_{L^2}^2 \int_{\Omega_1 \times \Omega_2} \phi_1^2(x)  K(x-y) dxdy +\frac{1}{2}\| \phi_0^{(2)}\|_{L^2}^2 \int_{\Omega_2 \times \Omega_1} \phi_2^2(x)  K(x-y) dxdy.
\end{array}
$$
We next note that
$$
\int_{\Omega_i} [R(x,0)-R(x,\theta_i)  ] \phi_0^{2}(x) dx= -g\theta_i\int_{\Omega_i} ( 2x-\theta_i) \phi_0^{2}(x) dx\leq 
-g(d^2-a^2/4) \int_{\Omega_i}\phi_0^{2}(x) dx.
$$
We also have
$$
 \int_{\Omega_i \times \Omega_j} \phi_i^2(x)  K(x-y) dxdy\leq a \|K\|_{L^\infty(\R)} \int_{\Omega_i } \phi_i^2(x)dx=a \|K\|_{L^\infty(\R)}.
 $$
 We deduce that
 $$
 \| \phi_0^{(1)}\|_{L^2}^2\mathcal R(\theta_1,\rho,\phi_1) + \| \phi_0^{(2)}\|_{L^2}^2\mathcal R(\theta_2,\rho,\phi_2)\leq\mathcal R(0,\rho,\phi_0)
 -g(d^2-a^2/4)+\frac a 2\|K\|_{L^\infty(\R)}  .
 $$
 Therefore, for $d$ large enough, we obtain that
 $$
  \| \phi_0^{(1)}\|_{L^2}^2\mathcal R(\theta_1,\rho,\phi_1) + \| \phi_0^{(2)}\|_{L^2}^2\mathcal R(\theta_2,\rho,\phi_2)<\mathcal R(0,\rho,\phi_0).
$$
We conclude that $\theta=0$ is not an emergent trait.

It remains to show that the population density cannot be monomorphic. Indeed, if we assume that $n$ is monomorphic  with $\overline{\theta}$ as the only emergent trait, then since the domain $\Omega$ is symmetric and according to Proposition \ref{prop:mono-uniq}, it follows that $\overline{\theta} = 0$. This is in contradiction with the fact that $\theta=0$ is not an emergent trait.
\end{proof}

\section{Some numerical illustrations}\label{discussion}

In this section, we illustrate the numerical solutions of \eqref{E}, for some particular examples, considering the following type of growth rate
\[R(x,\theta) = r - g (bx-\theta)^2.\]
We recall from Example \ref{example1} that in this case, $r$ is a maximal growth rate and $-g(bx-\theta)^2$ models the selection. The parameter $g$ is the selection pressure whereas $b$ is the gradient of the environment. We provide numerical examples where we vary this set of parameters. 

To find a numerical solution of \eqref{E}, we solve numerically the following parabolic equation:
\begin{equation}\label{Et} \tag{$E_t$}
    \left\lbrace
    \begin{aligned}
    &\partial_t n_\varepsilon -\partial_{xx} n_\varepsilon - \varepsilon^2 \partial_{\theta \theta} n_\varepsilon + L (n_\varepsilon) = [R - \rho_\varepsilon ] n_\varepsilon && \text{ in } \mathbb{R}^+ \times \Omega \times ]-A,A[, \\
    &\rho_\varepsilon (t,x) = \int_{-A}^A n_\varepsilon (t,x,\theta) d \theta && \text{ in } \mathbb{R}^+ \times \Omega, \\
    &\partial_{\nu_x } n_\varepsilon = \partial_{\nu_\theta} n_\varepsilon  = 0 , \\
    & n(t=0, x, \theta) = n_0(x,\theta).
    \end{aligned}
    \right.
\end{equation}
We implement equation \eqref{Et} by a semi-implicit finite difference method. We stop the algorithm when we find a numerical steady state of \eqref{Et}: a numerical solution of \eqref{E}.

\vspace{0.5cm}

First, we underline that in all the numerical resolutions, the density of the population concentrates around one or several distinct trait(s). Moreover, these emergent traits are present everywhere in space thanks to the local and the non-local migration. However, the density of the population at the position $x$ with a emergent trait $\theta_{m}$ depends on whether this trait $\theta_m$ is adapted or not to the position $x$.\\

\begin{figure}[!h]
  \begin{minipage}[b]{0.3\linewidth}
   \centering
   \includegraphics[width=5.5cm,height=5cm]{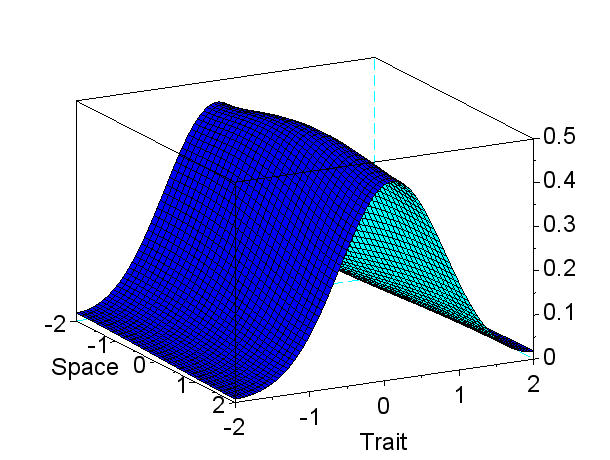} 
{$\varepsilon=0.1$}
  \end{minipage}
\hfill
\begin{minipage}[b]{0.3\linewidth}
   \centering
   \includegraphics[width=5.5cm,height=5cm]{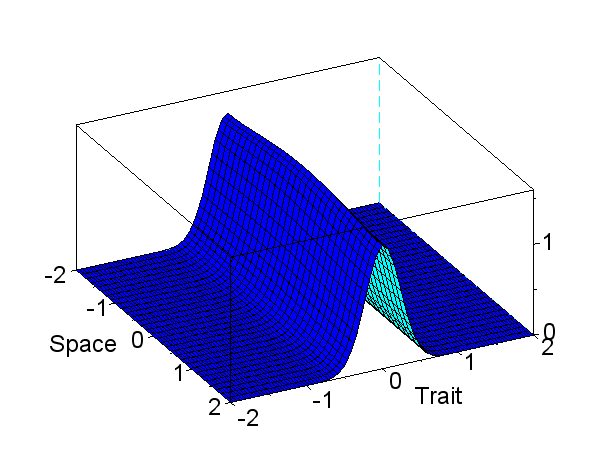} 
{$\varepsilon=0.01$}
  \end{minipage}
 \hfill
  \begin{minipage}[b]{0.3\linewidth}
   \centering
   \includegraphics[width=5.5cm,height=5cm]{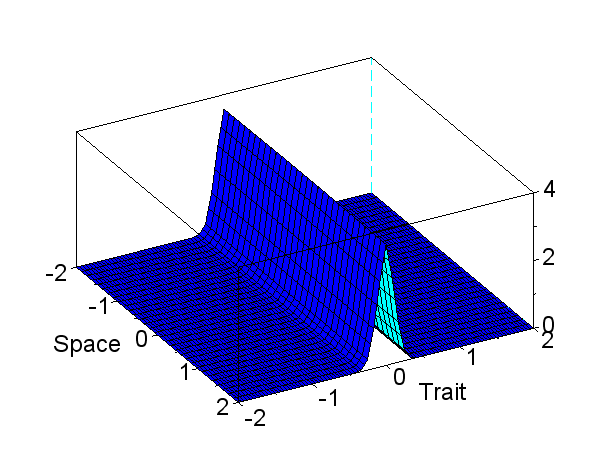}  
    {$\varepsilon=0.001$}
  \end{minipage}
  \caption{Variation of the numerical solutions of \eqref{E} with respect to $\varepsilon$. The others parameters are fixed as follow : $r=1$, $b=1$ $g=0.1$, $\Omega= ]-2,2[$ and $A=2$. We observe that the distribution of the population concentrates around the emergent trait $\theta = 0$.}
  \label{Fig1}
\end{figure}

Figure \ref{Fig1} illustrates the convergence of $n_\varepsilon$ to a Dirac mass as $\varepsilon$ goes to $0$. The only variation is with respect to the parameter $\varepsilon = 0.1, \ 0.01 \text{ and } 0.001$.

\begin{figure}[!h]
  \begin{minipage}[b]{0.3\linewidth}
   \centering
   \includegraphics[width=6cm,height=5cm]{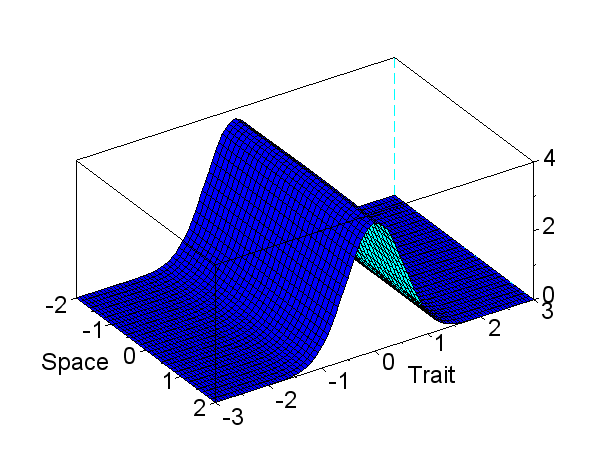} 
    {$g=0.01$}
  \end{minipage}
\hfill
  \begin{minipage}[b]{0.3\linewidth}
   \centering
   \includegraphics[width=6cm,height=5cm]{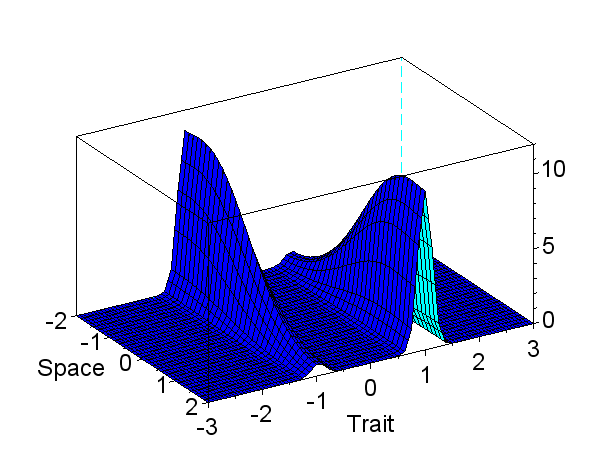} 
    {$g=1$}
  \end{minipage}
\hfill
\begin{minipage}[b]{0.3\linewidth}
   \centering
   \includegraphics[width=6cm,height=5cm]{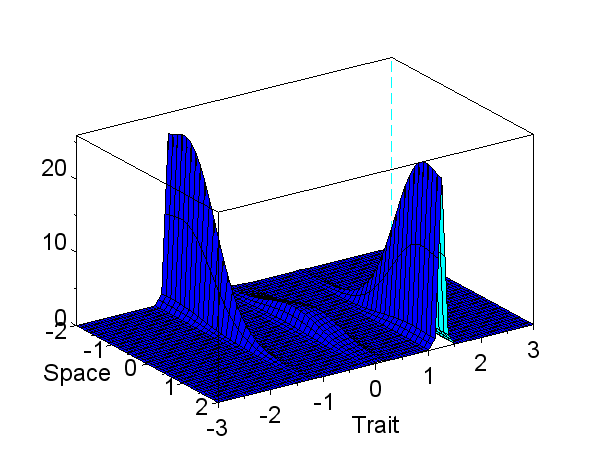} 
    {$g=5$}
  \end{minipage}
  \caption{Variation of the numerical solutions $n_\varepsilon$ of \eqref{E} with respect to $g$. The other parameters are fixed as follow:  $r=5$, $b=1$, $\varepsilon=0.01$, $\Omega= ]-2,2[$ and $A=3$. We recover that if $g$ is small then the population is monomorphic. For large values of $g$,   there exist several distinct emergent traits.}
  \label{Fig2}
\end{figure}

\begin{figure}[!h]
  \begin{minipage}[b]{0.3\linewidth}
   \centering
     \includegraphics[width=6cm,height=5cm]{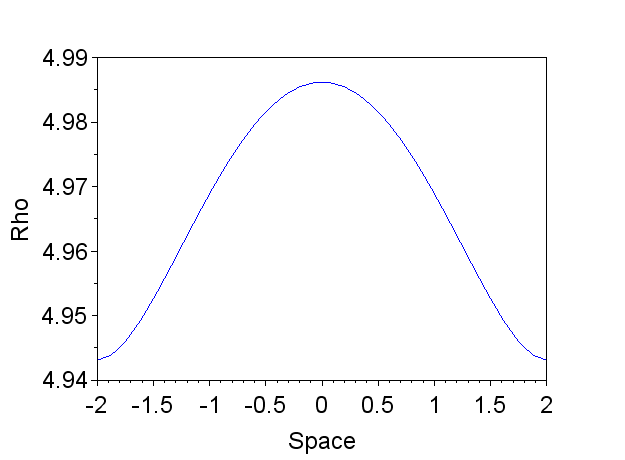} 
      {$g=0.01$}
  \end{minipage}
\hfill
  \begin{minipage}[b]{0.3\linewidth}
   \centering
     \includegraphics[width=6cm,height=5cm]{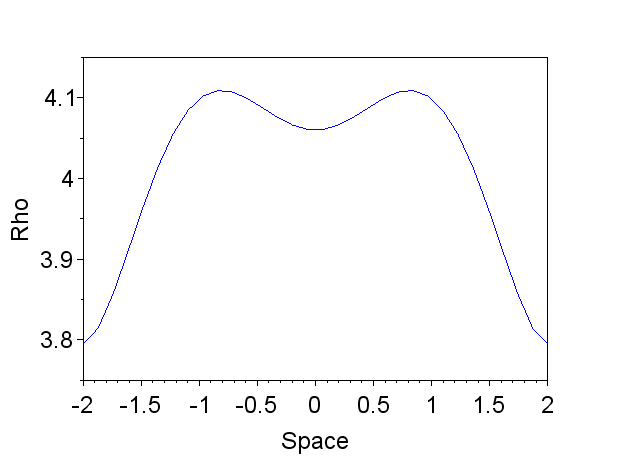} 
      {$g=1$}
  \end{minipage}
\hfill
\begin{minipage}[b]{0.3\linewidth}
   \centering
   \includegraphics[width=6cm,height=5cm]{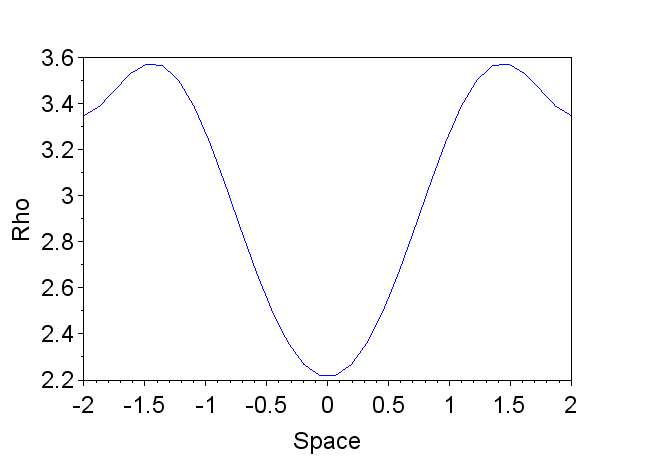} 
{$g=5$}
  \end{minipage}
  \caption{Variation of the numerical density $\rho_\varepsilon$ of \eqref{E} with respect to $g$. The other parameters are fixed as follow:  $r=5$, $b=1$, $\varepsilon=0.01$, $\Omega=]-2,2[$ and $A=3$.}
  \label{Fig22}
\end{figure}

Next, we focus on the qualitative properties established in Section \ref{sec:qual}, Figures \ref{Fig2} and \ref{Fig22} are numerical illustrations of Proposition \ref{prop:mono-uniq}. We fix $\Omega$ as a single connected component and we investigate the dependence on the parameter $g$. We recover numerically that as $g \to 0$ the limit density is monomorphic with  an emergent trait 
 at $\theta = 0$. For larger values of $g$, the phenotypic density concentrates around several distinct traits. For each simulations, we also provide the numerical distributions of $\rho_\varepsilon$ (Figure \ref{Fig22}), this density seems to be centered around the point which maximizes $R(\cdot, \overline{\theta})$ (where $\overline{\theta}$ is any emergent trait). Therefore, when the emergent trait is unique, $\rho_\varepsilon$ is increasing on $]-2, 0[$ and then decreasing on $ ]0, 2[$ whereas the spatial distribution can be more involved whenever there exist several distinct emergent traits.

\begin{figure}[!h]
  \begin{minipage}[b]{0.45\linewidth}
   \centering
   \includegraphics[width=8cm,height=6cm]{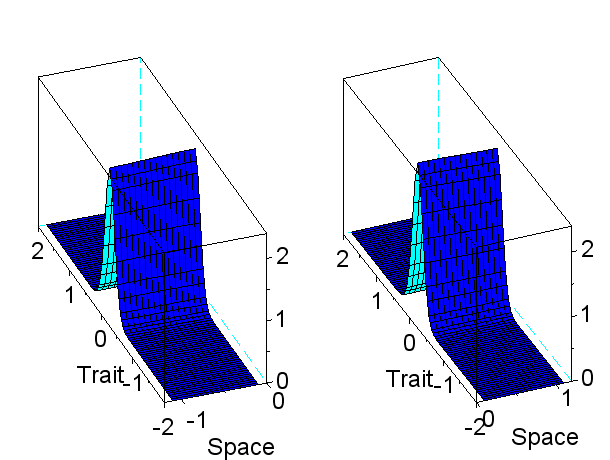} 
    {$\Omega=]-1.1, -0.1[ \cup ]0.1, 1.1[$}
  \end{minipage}
\hfill
\begin{minipage}[b]{0.45\linewidth}
   \centering
   \includegraphics[width=8cm,height=6cm]{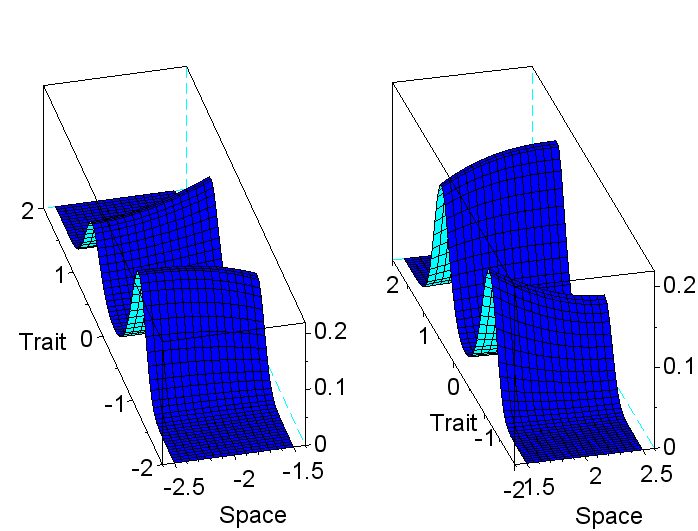} 
    {$\Omega = ]-2.5, -1.5[ \cup ]1.5, 2.5[$}
  \end{minipage}
  \caption{Variation of the numerical solutions $n_\varepsilon$ of \eqref{E} with respect to the distance between the two connected components of $\Omega$. The other parameters are fixed as follow:  $r=1$, $b=1$, $g=1$ $\varepsilon=0.01$ and $A=2$. We observe that for this set of parameters, increasing the distance   induces a polymorphic density population.}
  \label{Fig3}
\end{figure}

To conclude, we present in Figure \ref{Fig3} a numerical illustration of Proposition \ref{prop:polydist}. Here, the free parameter is the distance between the two connected components of $\Omega$. We recover that increasing the  distance between the two connected components  may lead to multiple emergent  traits.

\section{Existence of a non-trivial solution of \eqref{E}}\label{sec:existence}

As mentioned in the introduction, we recall that the proof of existence of a non-trivial solution is an adaptation of the proof of Theorem 2.1 by Lam and Lou in \cite{LamLou}. The major difference is the presence of the integral operator $L$. Therefore, we only provide the main elements dealing with the integral operator $L$. We also skip the proof of non-existence of a non-trivial solution, when \eqref{assumption} does not hold, which   follows   from classical arguments.

\begin{proof}[Proof of Theorem \ref{existence}]

We fix $\varepsilon \in  ]0, \varepsilon_0[$ (where $\varepsilon_0$ is given by \eqref{H5}). Let $\tau \in [0,1]$ and $n_\tau$ be a solution of 
\begin{equation} \label{Etau} \tag{$E_\tau$}
\left\lbrace 
\begin{aligned}
& -\partial_{xx} n_\tau - \varepsilon^2 \partial_{\theta \theta} n_\tau + L n_\tau = n_\tau \left(R - \tau \rho_\tau - (1-\tau) n_\tau \right) \qquad  \text{ in } \Omega \times ]-A, A[, \\
&\rho_\tau (x) = \int_{]-A,A[} n_\tau (x,\theta) d \theta \qquad \text{ in } \Omega, \\
&\partial_{\nu_x} n_\tau (x, \theta )= 0 \ \text{ on } \partial \Omega \times ]-A,A[, \quad \partial_{\nu_\theta} n_\tau (x, \pm A) = 0 \text{ on } \Omega \times \left\lbrace \pm A \right\rbrace. 
\end{aligned}
\right.
\end{equation}
It is well known that for $\tau = 0$, according to \eqref{H5}, there exists a non-trivial steady solution $n_0$. As in \cite{LamLou}, we prove that there exists a constant $C_\varepsilon>1$ (which may depend on $\varepsilon$) such that we have for any $\tau \in [0,1]$
\[ C_\varepsilon^{-1} \leq \int_{-A}^A \int_\Omega n_\tau dx d\theta  \leq C_\varepsilon. \]
Then one can conclude using a topological degree arguments. 

\bigbreak

\textbf{The lower bound. } Let $v_\tau$ be such that $n_\tau = \xi_\varepsilon v_\tau$ (where $\xi_\varepsilon$ is provided by \eqref{vpchap52}). First, we remark that 
\begin{align*}
     &L(v_\tau \xi_\varepsilon ) = v_\tau L(\xi_\varepsilon) + \xi_\varepsilon L(v_\tau) + \Lambda(v_\tau , \xi_\varepsilon) \\
     \text{ and } \quad & \Lambda (v_\tau , \xi_\varepsilon) (x) = \int_\Omega [ (v_\tau (x) - v_\tau(y) ) ( \xi_\varepsilon (y) - \xi_\varepsilon (x)) ] K(x-y) dy.
\end{align*}
Then $v_\tau$ is solution of 
\begin{align*} -\xi_\varepsilon \partial_{xx} v_\tau -2 \partial_x \xi_\varepsilon \partial_{x} v_\tau  - \varepsilon^2 \xi_\varepsilon \partial_{\theta \theta} v_\tau -2 \varepsilon^2 \partial_\theta \xi_\varepsilon \partial_{\theta} v_\tau + \xi_\varepsilon L(v_\tau ) +&\Lambda ( v_\tau, \xi_\varepsilon) + \mu_\varepsilon \xi_\varepsilon v_\tau\\
&= -v_\tau \xi_\varepsilon  [\tau \rho_\tau + (1-\tau ) n_\tau ].
\end{align*}
If we multiply it by $\frac{\xi_\varepsilon}{v_\tau}$, we obtain
\[ \frac{- \partial_x (\xi_\varepsilon^2 \partial_{x} v_\tau)  - \varepsilon^2 \partial_\theta (\xi_\varepsilon^2 \partial_{\theta} v_\tau) + \xi_\varepsilon^2 L ( v_\tau ) + \xi_\varepsilon \Lambda(v_\tau, \xi_\varepsilon)}{v_\tau }  =  \xi_\varepsilon ^2 (-\mu_\varepsilon - \tau \rho_\tau - (1-\tau ) n_\tau).\]
Next, we integrate over all the domain 
\begin{align*} 
&\int_{-A}^A\int_{\Omega}\frac{- \partial_x (\xi_\varepsilon^2 \partial_{x} v_\tau)  - \varepsilon^2 \partial_\theta (\xi_\varepsilon^2 \partial_{\theta} v_\tau) + \xi_\varepsilon^2 L ( v_\tau ) + \xi_\varepsilon \Lambda (v_\tau, \xi_\varepsilon)}{v_\tau} dx d\theta \\
&= \int_{-A}^A\int_{\Omega}\frac{- \partial_x (\xi_\varepsilon^2 \partial_{x} v_\tau)  - \varepsilon^2 \partial_\theta (\xi_\varepsilon^2 \partial_{\theta} v_\tau) }{v_\tau} dx d\theta+ \int_{-A}^A\int_{\Omega}\frac{ \xi_\varepsilon^2 L ( v_\tau  ) + \xi_\varepsilon \Lambda (v_\tau, \xi_\varepsilon) }{v_\tau} dx d\theta\\
&= I_1+I_2.
\end{align*}
We next prove that $I_1$ and $I_2$ are negative. For $I_1$, by an integration by part, we have 
\[ I_1 = \int_{-A}^A\int_{\Omega}\frac{- \partial_x (\xi_\varepsilon^2 \partial_{x} v_\tau)  - \varepsilon^2 \partial_\theta (\xi_\varepsilon^2 \partial_{\theta} v_\tau) }{v_\tau} dx d\theta = -\int_{-A}^A\int_{\Omega}\frac{ \xi_\varepsilon^2 }{v_\tau^2} \left(|\partial_{x} v_\tau|^2  + \varepsilon^2 |\partial_{\theta} v_\tau|^2\right) dx d\theta \leq 0.\]
For $I_2$, using that $K$ is even and the Fubini Theorem, we obtain 
\begin{align*}
 I_2&=\int_{-A}^A\int_{\Omega}\frac{ \xi_\varepsilon^2(x) L ( v_\tau  )(x) + \xi_\varepsilon(x) \Lambda(v_\tau, \xi_\varepsilon)(x) }{v_\tau(x)} dx d\theta \\
 &=  \int_{-A}^A \int_{\Omega} \frac{\xi_\varepsilon(x)}{v_\tau(x)} \int_\Omega [( v_\tau(x) - v_\tau(y)) \xi_\varepsilon(y)] K(x-y)dy dx  d\theta \\
 &=  -\int_{-A}^A \int_{\Omega}  \int_\Omega \left[\frac{ \xi_\varepsilon(x)\xi_\varepsilon(y) }{v_\tau (x) v_\tau(y)}( v_\tau(y) - v_\tau(x))^2 \right] K(x-y)dy dx d\theta - I_2
\end{align*}
We deduce that
\[I_2 = -\frac{1}{2} \int_{-A}^A \int_{\Omega}  \int_\Omega \left[\frac{ \xi_\varepsilon(x)\xi_\varepsilon(y) }{v_\tau (x) v_\tau(y)}( v_\tau(y) - v_\tau(x))^2 \right] K(x-y)dy dx  d\theta \leq 0.\]
Therefore, we have that 
\[\int_{-A}^A \int_\Omega \xi_\varepsilon^2 [ -\mu_\varepsilon - \tau \rho_\tau - (1-\tau) n_\tau] dx d\theta \leq 0.\]
Thanks to \eqref{H5}, we conclude that for $\varepsilon$ small enough
\[ \frac{|\lambda(\theta_0, 0)|}{2} \leq -\mu_\varepsilon = -\mu_\varepsilon \int_{-A}^A \int_\Omega \xi_\varepsilon^2 dxd\theta \leq \sup (\xi_\varepsilon^2) [\tau + (1-\tau)] \  \int_{-A}^A \int_{\Omega}n_\tau dx d\theta.\]

\bigbreak 

\textbf{The upper bound. } First, we remark that thanks to the Neumann boundary conditions and the parity of $K$, we have that 
\[ \int_{-A}^A \int_\Omega -\partial_{xx} \left(  n_\tau - \partial_{\theta \theta } n_\tau + L (n_\tau) \right) dx d\theta = 0. \]
Therefore, if we integrate \eqref{Etau} with respect to $x$ and $\theta$, we obtain 
\begin{align*}
\left( \frac{(1-\tau)}{2A|\Omega|} + \frac{\tau}{|\Omega|}\right) \ \| n_\tau \|_{L^1}^2& =  \frac{\tau}{|\Omega|} \left(\int_\Omega \rho_\tau dx \right)^2+ \frac{(1-\tau)}{2A|\Omega|} \left( \int_{-A}^A\int_\Omega n_\tau dx d\theta \right)^2 \\ &\leq  \tau \int_\Omega \rho_\tau^2 dx + (1-\tau)\int_{-A}^A \int_\Omega  n_\tau^2 dx d\theta \\
&= \int_{-A}^A \int_\Omega R n_\tau dx d\theta  \leq C_R \int_{-A}^A \int_\Omega  n_\tau dx d\theta  =C_R  \ \|n_\tau \|_{L^1}.
\end{align*}

\bigbreak

\textbf{Conclusion. } It follows that there exists a bounded non trivial solution $n_\varepsilon$ of \eqref{E}. Moreover, we have indeed proved that there exists constants $c,C>0$ such that 
\[ \frac{c}{\sup \xi_\varepsilon^2 } \leq \int_{-A}^A \int_\Omega n_\varepsilon dx d\theta \leq C.\]

\end{proof}

\section{Regularity results}\label{regularity}

In this section we prove Theorem \ref{Regularitythm}. The sub-sections correspond respectively to the proof of the item  \textit{1. 2. 3.} and \textit{4.} of Theorem \ref{Regularitythm}. But, we need an intermediate result: $\rho_\varepsilon$ is uniformly bounded.
\begin{lemma}\label{bdrho}
Under the assumptions \eqref{H1Chap5} -- \eqref{assumption}, we have that for all $\varepsilon < \varepsilon_0$
\[0 \leq \rho_\varepsilon \leq C_R,\]
(where $C_R$ is introduced in \eqref{H2Chap5}). Moreover, there exists $C>0$ such that for all $\varepsilon$ small enough
\[ \| \rho_\varepsilon \|_{W^{2,p}(\Omega)} \leq C.\]

\end{lemma}

\begin{proof}
\textbf{The $L^\infty$-bounds. }   It is obvious that $\rho_\varepsilon > 0$. If we integrate \eqref{E} with respect to $\theta$, we obtain
\begin{equation}\label{eqrho}\tag{$E_\rho$}
    \left\lbrace
    \begin{aligned}
    & -\partial_{xx} \rho_\varepsilon + L \rho_\varepsilon= \int_{-A}^A R(\cdot, \theta) n_\varepsilon (\cdot, \theta) d\theta - \rho_\varepsilon^2 && \text{ in } \Omega, \\
    & \partial_{\nu_x} \rho_{\varepsilon} = 0 \text{ in } \partial \Omega.
    \end{aligned}
    \right.
\end{equation}
Recalling the $L^\infty$ bounds on $R$ \eqref{H2Chap5}, it follows:
\[ -\partial_{xx} \rho_\varepsilon +  L \rho_\varepsilon \leq C_R \rho_\varepsilon - \rho_\varepsilon^2.\]
We conclude thanks to the maximum principle that $\rho_\varepsilon \leq C_R$.

\vspace{0.5cm}

\noindent \textbf{The $W^{2,p}(\Omega)$ bounds.}  Thanks to the $L^\infty$ bounds on $R, K, \rho_{\varepsilon}$ (assumptions \eqref{H2Chap5}, \eqref{K} and the previous inequality), we may write \eqref{eqrho} on the following form
\[ -\partial_{xx} \rho_\varepsilon = f_\varepsilon\]
with $f_\varepsilon \in L^\infty(\Omega)$ uniformly bounded. The result follows from the standard elliptic estimates. 

\end{proof}

\begin{corollary}\label{CorollaryRho}
There exists a constant $C>0$ such that for all $\varepsilon$ small enough
\begin{equation}\label{diffrho}
| \partial_x \rho_\varepsilon | \leq C. 
\end{equation}
\end{corollary}

\subsection{A Harnack inequality}

The first step to prove the first item of Theorem \ref{Regularitythm} is to prove the result in the interior of $\Omega \times ]-A,A[$.

\begin{theorem}\label{Harncak1}
For all $(x_0, \theta_0) \in \Omega \times ]-A,A[$, and $R_0 >0$ such that 
\[B_{3R_0} (x_0) \times B_{3\varepsilon R_0}( \theta_0) \subset \Omega \times ]-A, A[\]
there exists $C(R_0)>0$ such that 
\begin{equation}\label{in}
\underset{(x, \theta ) \in  B_{R_0} (x_0) \times B_{\varepsilon R_0}(\theta_0) }{\sup} n_\varepsilon (x,\theta) \leq C(R_0) \underset{ (x, \theta) \in  B_{R_0} (x_0) \times B_{\varepsilon R_0}( \theta_0) }{\inf} n_\varepsilon(x, \theta). 
\end{equation}
\end{theorem}
Next, we prove that we can extend the solution thanks to a reflective argument(see Remark 9 p.275 in \cite{brezis}). \\ 
We perform the following change of variable: $\widetilde{n}(x,\theta) = n_\varepsilon(x, \varepsilon \theta)$. Therefore, we consider the following scaled equation 

\begin{equation}\label{E'} \tag{E'}
    \left\lbrace
    \begin{aligned}
    &-\partial_{xx} \widetilde{n} - \partial_{\theta\theta}\widetilde{n}  + L\widetilde{n} = \widetilde{n} [\widetilde{R} - \rho] && \text{ in } \Omega \times ] -\varepsilon^{-1} A , \varepsilon A[, \\
    &\partial_{\nu_x} \widetilde{n} = \partial_{\nu_\theta} \widetilde{n} =0 && \text{ in } \partial( \Omega \times ]-\varepsilon^{-1}A, \varepsilon A[.
    \end{aligned}
    \right.
\end{equation}
We have denoted by $\widetilde{R}$ the function $\widetilde{R}(x, \theta) = R(x, \varepsilon \theta)$. Remark that $\widetilde{R}$ still verifies \eqref{H2Chap5}.

\begin{proof}[Proof of Theorem \ref{Harncak1}]
Let $(x_0, \theta_0) \in \Omega \times ]-\varepsilon^{-1} A, \varepsilon^{-1}A[$ and a radius $R_0 >0$ be such that $B_{3R_0} (x_0, \theta_0) \subset \Omega \times ]-\varepsilon^{-1}A, \varepsilon^{-1}A[$.  If we denote by $f(x,\theta) = \int_\Omega \widetilde{n}(y, \theta) K(y-x)dy$, according to \eqref{K} it follows that $f \in L^\infty (B_{2R_0}(x_0, \theta_0))$. From the classical Harnack inequality, the Theorem 9.20 and 9.22 pp. 244-246 in \cite{Trudinger}, and using \eqref{H2Chap5} we deduce the existence of $C_1>0$ (depending on $R_0$) such that 
\begin{equation} \label{PfHarnack1}
\begin{aligned}
\underset{(x,\theta) \in B_{R_0}(x_0, \theta_0) }{\sup} \widetilde{n}(x, \theta) &\leq C_1 \underset{(x,\theta') \in B_{R_0}(x_0, \theta_0) }{\inf} \widetilde{n}(x,\theta') + C_1 \underset{(x,\theta") \in B_{2R_0}(x_0, \theta_0) }{\sup} |f(x,\theta")| \\
& \leq C_1 \underset{(x,\theta') \in B_{R_0}(x_0, \theta_0) }{\inf} \widetilde{n}(x,\theta')  + C_1 C_K \underset{\theta" \in B_{2R_0}(\theta_0)}{\sup } \int_{\Omega} \widetilde{n}(x, \theta") dx.
\end{aligned}
\end{equation} 
The main element of the proof is to prove the following claim:
\begin{equation}\label{AimHarnack}
\exists C>0 \quad \text{ such that } \quad \underset{\theta \in B_{2R_0}(\theta_0)}{\sup } \int_{\Omega} \widetilde{n}(x, \theta) dx  \leq C \underset{(x,\theta) \in B_{R_0}(x_0, \theta_0)}{\inf} \widetilde{n}(x,\theta).
\end{equation}
It is clear that if \eqref{AimHarnack} holds true, the conclusion follows. 

\bigbreak

First, we integrate \eqref{E'} with respect to $x$. It follows thanks to the Neumann boundary conditions that for all $\theta \in B_{3R_0}(\theta_0)$ we have
\begin{align*}
 - \partial_{\theta \theta } \int_\Omega \widetilde{n}(x,\theta) dx =& \frac{\int_\Omega \widetilde{n} (x,\theta) \left( \widetilde{R}(x,\theta) - \rho(x) - \int_\Omega K(x-y)dy \right) dx}{\int_\Omega \widetilde{n}(x,\theta) dx} \int_\Omega \widetilde{n}(x,\theta) dx \\
& + \frac{\int_{\Omega} \int_{\Omega} \widetilde{n}(y, \theta) K(x-y) dydx}{\int_\Omega \widetilde{n}(x,\theta) dx} \int_\Omega \widetilde{n}(x,\theta) dx.
\end{align*}
Thanks to the $L^\infty$-bounds on $ K,\widetilde{R} ,\rho$ (assumptions \eqref{H2Chap5}, \eqref{K} and Lemma \ref{bdrho}) and the Fubini Theorem, we have 
\begin{align*}
&-C \leq \frac{\int_\Omega \widetilde{n} (x,\theta) \left( \widetilde{R}(x,\theta) - \rho(x) - \int_\Omega K(x-y)dy \right) dx}{\int_\Omega \widetilde{n}(x,\theta) dx}< C\\
 \text{ and }\quad & \quad \frac{\int_{\Omega} \int_{\Omega} \widetilde{n}(y, \theta) K(x-y) dydx}{\int_\Omega \widetilde{n}(x,\theta) dx} \leq C_K |\Omega|.
\end{align*}
It follows 
\[-C \int_{\Omega } \widetilde{n}(x, \theta)dx \leq -\partial_{\theta \theta } \int_{\Omega} \widetilde{n}(x,\theta)dx \leq C \int_{\Omega } \widetilde{n} (x, \theta)dx .\]
Hence, we apply the Harnack inequality to $\theta \in B_{3R_0}(\theta_0) \mapsto \int_\Omega \widetilde{n}(x,\theta) dx$ into the ball $B_{2R_0}(\theta_0)$ and we deduce the existence of a constant $C_2>0$ such that 
\begin{equation}\label{PfHarnack2}
\underset{ \theta \in B_{2R_0}(\theta_0) }{\sup } \int_\Omega \widetilde{n}(x,\theta) dx \leq C_2 \underset{ \theta \in B_{2R_0}(\theta_0) }{\inf} \int_\Omega \widetilde{n}(x,\theta) dx.
\end{equation}
Next, thanks to the $L^\infty$-bounds on $ K, \widetilde{R}, \rho$ (assumptions \eqref{H2Chap5}, \eqref{K} and Lemma \ref{bdrho}), it follows that in $\Omega \times B_{2R_0}(\theta_0)$
\[c_K \underset{\theta \in B_{2R_0}(\theta_0)}{\inf} \int_{\Omega} \widetilde{n}(y, \theta) dy \leq c_K \int_{\Omega} \widetilde{n}(y, \theta) dy  \leq  \int_{\Omega} \widetilde{n}(y, \theta) K(x-y) dy   \leq \left(-\partial_{xx} -  \partial_{\theta \theta} \right) \widetilde{n} + C \widetilde{n}  .\]
From an inequality developed by Krylov (we refer to Theorem 7.1 p.565 in \cite{Harnack} and the reference therein), we deduce the existence of a constant $C_3>0$ such that 
\begin{equation}\label{PfHarnack3}
 \underset{B_{2R_0}(\theta_0)}{\inf} \int_{\Omega} \widetilde{n}(x, \theta) dx \leq C_3 \underset{B_{R_0}(x_0, \theta_0)}{\inf}  \widetilde{n}(x, \theta) .
\end{equation}
Combining the previous inequality with \eqref{PfHarnack2} and \eqref{PfHarnack3} yields to 
\begin{equation*}
 \underset{\theta \in  B_{2R_0}( \theta_0) }{\sup} \int_\Omega \widetilde{n}(x,\theta) dx \leq C_2 \underset{\theta \in  B_{2R_0}( \theta_0) }{\inf} \int_\Omega \widetilde{n}(x,\theta) dx \leq C_2 C_3 \underset{(x,\theta) \in  B_{R_0}(x_0, \theta_0) }{\inf}  \widetilde{n}(x,\theta)   .
\end{equation*}
This concludes the proof. 
\end{proof}

\subsection{Lipschitz estimates}

We prove \textit{2.} of Theorem \ref{Regularitythm} by the Bernstein method.

\begin{proof}[Proof of \textit{2.} of Theorem \ref{Regularitythm}]
We recall the main equation satisfied by $u_\varepsilon$:
\begin{equation}\label{PfLip1}
\frac{ -\partial_{xx} u_\varepsilon }{\varepsilon } - \frac{ |\partial_x u_\varepsilon |^2}{\varepsilon^2} - \varepsilon \partial_{\theta \theta }u_\varepsilon - |\partial_\theta u_\varepsilon|^2 + \int_{\Omega } [1- e^{\frac{u_\varepsilon(y) - u_\varepsilon(x)}{\varepsilon}} ] K(x-y)dy = R(x,\theta) - \rho_\varepsilon
\end{equation}
with Neumann boundary conditions. The first step is to differentiate \eqref{PfLip1} with respect to $x$ and multiply it by $\frac{\partial_x u_\varepsilon}{\varepsilon^2}$:
\begin{equation*}
\begin{aligned}
   & -\frac{\partial_{xxx} u_\varepsilon \partial_x u_\varepsilon}{\varepsilon^3}  - \frac{ \partial_x \left( \frac{ | \partial_x u_\varepsilon |^2}{\varepsilon^2} \right) \partial_x u_\varepsilon}{\varepsilon^2 } + \int_{\Omega} e^{\frac{ u_\varepsilon(y) - u_\varepsilon(x)}{\varepsilon}} K(x-y)dy  \frac{ \partial_x u_\varepsilon ^2}{\varepsilon^3} \\
    & - \frac{\partial_x |\partial_\theta u_\varepsilon |^2 \partial_x u_\varepsilon}{\varepsilon^2}  - \frac{ \partial_{x\theta\theta} u_\varepsilon \partial_x u_\varepsilon }{\varepsilon}   = \frac{ \left(\int_{\Omega} [e^{\frac{u_\varepsilon(y)-u_\varepsilon(x)}{\varepsilon}}-1] \partial_x K(x-y)dy + \partial_x R  - \partial_x \rho_\varepsilon \right)\partial_x u_\varepsilon }{\varepsilon^2}. 
\end{aligned}
\end{equation*}
Remarking that
\[\partial_{xxx} u_\varepsilon \partial_x u_\varepsilon = \frac{ \partial_{xx} (|\partial_x u_\varepsilon|^2)}{2} - (\partial_{xx} u_\varepsilon )^2  \qquad \text{ and } \qquad 
\partial_{x\theta\theta} u_\varepsilon \partial_x u_\varepsilon = \frac{ \partial_{\theta\theta} (|\partial_x u_\varepsilon|^2)}{2} - (\partial_{\theta x} u_\varepsilon )^2\]
yields to 
\begin{equation}\label{dxPfLip1}
\begin{aligned}
    &-\frac{\partial_{xx} (\frac{|\partial_x u_\varepsilon |^2}{\varepsilon^2})}{2\varepsilon} + \frac{(\partial_{xx} u_\varepsilon)^2}{\varepsilon^3}   - \frac{ \partial_x \left( \frac{ | \partial_x u_\varepsilon |^2}{\varepsilon^2} \right) \partial_x u_\varepsilon}{\varepsilon^2 }  \\
    & + \int_{\Omega }    e^{\frac{ u_\varepsilon(y) - u_\varepsilon(x)}{\varepsilon}} K(x-y)dy \frac{ ( \partial_x u_\varepsilon) ^2}{\varepsilon^3} - \frac{\partial_x |\partial_\theta u_\varepsilon |^2 \partial_x u_\varepsilon}{\varepsilon^2}  \\
    &+ \frac{ (\partial_{\theta x} u_\varepsilon )^2}{\varepsilon} -\frac{ \varepsilon \partial_{\theta\theta} (\frac{|\partial_x u_\varepsilon|^2}{\varepsilon^2})}{2}=\frac{\left( \int_{\Omega} [e^{\frac{u_\varepsilon(y)-u_\varepsilon(x)}{\varepsilon}}-1] \partial_x K(x-y)dy +  \partial_x R  - \partial_x \rho_\varepsilon \right) \partial_x u_\varepsilon }{\varepsilon^2}. 
\end{aligned}
\end{equation}
In the second step, we differentiate \eqref{PfLip1} with respect to $\theta$ and multiply by $\partial_\theta u_\varepsilon$. With computations similar to the ones presented above, we find 
\begin{equation}\label{dthetaPfLip1}
\begin{aligned}
    -\frac{\partial_{xx} (|\partial_\theta u_\varepsilon |^2)}{2\varepsilon} &+ \frac{(\partial_{\theta x} u_\varepsilon)^2}{\varepsilon} -  \partial_\theta \frac{ | \partial_x u_\varepsilon |^2}{\varepsilon^2} \partial_\theta  u_\varepsilon - \frac{\varepsilon}{2} \partial_{\theta\theta} (|\partial_\theta u_\varepsilon|^2) + \varepsilon (\partial_{\theta \theta} u_\varepsilon )^2  - \partial_\theta |\partial_\theta u_\varepsilon |^2 \partial_\theta u_\varepsilon \\
    &+ \int_{\Omega }\frac{ \left(  \partial_\theta u_\varepsilon (x)^2- \partial_\theta u_\varepsilon (x) \partial_\theta u_\varepsilon(y) \right)}{\varepsilon} e^{\frac{ u_\varepsilon(y) - u_\varepsilon(x)}{\varepsilon}} K(x-y)dy = \partial_\theta R  \partial_\theta u_\varepsilon . 
\end{aligned}
\end{equation}
Next, we introduce 
\begin{equation}\label{peps}
    p_\varepsilon(x,\theta) = \frac{ |\partial_x u_\varepsilon (x,\theta) |^2}{\varepsilon^2} + |\partial_\theta u_\varepsilon(x,\theta) |^2.
\end{equation}
If we combine \eqref{dxPfLip1} and \eqref{dthetaPfLip1}   and we rewrite it in terms of $p_\varepsilon$ , it follows 
\begin{equation}\label{pPfLip1}
\begin{aligned}
&-\frac{\partial_{xx} p_\varepsilon }{2\varepsilon } -\frac{\varepsilon \partial_{\theta \theta} p_\varepsilon }{2} + \frac{1}{\varepsilon}\int_{\Omega} [p_\varepsilon (x,\theta) -\partial_{\theta}u_\varepsilon(x, \theta)\partial_{\theta}u_\varepsilon (y, \theta )]e^{\frac{u_\varepsilon(y, \theta) - u_\varepsilon(x, \theta)}{\varepsilon }} K(x-y)dy\\
& - \frac{ \partial_x p_\varepsilon \partial_x u_\varepsilon}{\varepsilon^2} - \partial_\theta p_\varepsilon \partial_\theta u_\varepsilon  + \frac{2 (\partial_{x\theta} u_\varepsilon)^2}{\varepsilon} + \frac{(\partial_{xx}u_\varepsilon)^2 }{\varepsilon^3 } + \varepsilon (\partial_{\theta \theta } u_\varepsilon)^2 \\
&=\left(  \int_{\Omega} [e^{\frac{u_\varepsilon(y)-u_\varepsilon(x)}{\varepsilon}}-1] \partial_x K(x-y)dy  + \partial_x R -\partial_x\rho_\varepsilon  \right) \frac{\partial_x u_\varepsilon }{\varepsilon^2} + \partial_\theta R \partial_\theta u_\varepsilon.
\end{aligned}
\end{equation}
Let $(x_\varepsilon, \theta_\varepsilon ) $ be such that 
\[\underset{ (x,\theta) \in \Omega \times ]-A,A[}{\sup} p_\varepsilon (x, \theta) = p_\varepsilon (x_\varepsilon, \theta_\varepsilon) .\]
Thanks to the Neumann boundaries conditions, we deduce that $(x_\varepsilon, \theta_\varepsilon) \notin \partial\Omega \times \partial \left(]-A,A[\right)$. Therefore, we distinguish three cases: either $(x_\varepsilon, \theta_\varepsilon) \in \Omega \times ]-A,A[$ or $(x_\varepsilon, \theta_\varepsilon) \in \partial \Omega\times ]-A,A[$ or $(x_\varepsilon, \theta_{\varepsilon}) \in \Omega \times \left\lbrace \pm A \right\rbrace$.\\
\textbf{Case 1 : $(x_\varepsilon, \theta_\varepsilon ) \in \Omega \times ]-A,A[.$ } First, we bound the right-hand-side of \eqref{pPfLip1}. Indeed, thanks to the Harnack inequality (first item of Theorem \ref{Regularitythm}) and the $L^\infty$-bounds on the derivative of $K$, $R$ and $\rho_\varepsilon$ (assumptions  \eqref{H2Chap5}, \eqref{K} and \eqref{diffrho} in Corollary \ref{CorollaryRho}), it follows that
\begin{equation}\label{RHS}
\left(  \int_{\Omega} [e^{\frac{u_\varepsilon(y)-u_\varepsilon(x)}{\varepsilon}}-1] \partial_x K(x-y)dy  + \partial_x R -\partial_x \rho_\varepsilon \right) \frac{\partial_x u_\varepsilon }{\varepsilon^2} + \partial_\theta R \partial_\theta u_\varepsilon \leq \frac{C \sqrt{p}}{\varepsilon}.
\end{equation}
Next, we evaluate \eqref{pPfLip1} at $(x_\varepsilon, \theta_\varepsilon)$. We claim that
\begin{equation}\label{LHS}
\begin{aligned}
&-\partial_{xx}p_\varepsilon (x_\varepsilon , \theta_\varepsilon ) \geq 0,   \quad -\partial_{\theta \theta}p_\varepsilon (x_\varepsilon , \theta_\varepsilon ) \geq 0, \quad \partial_x p_\varepsilon (x_\varepsilon, \theta_\varepsilon) = \partial_\theta p_\varepsilon (x_\varepsilon, \theta_\varepsilon)=0 \\
\text{ and } \quad &  \frac{1}{\varepsilon}\int_{\Omega} [p_\varepsilon (x_\varepsilon,\theta_\varepsilon) -\partial_{\theta}u_\varepsilon(x_\varepsilon, \theta_\varepsilon)\partial_{\theta}u_\varepsilon (y, \theta_\varepsilon )]e^{\frac{u_\varepsilon(y, \theta_\varepsilon) - u_\varepsilon(x_\varepsilon, \theta_\varepsilon)}{\varepsilon }} K(x_\varepsilon-y)dy\geq 0.
\end{aligned}
\end{equation}
Indeed, the first inequalities follow easily since $p(x_\varepsilon, \theta_\varepsilon) = \max \  p_\varepsilon$ and the last inequality holds true thanks to the following computations
\begin{align*}
& \frac{1}{\varepsilon}\int_{\Omega} [p_\varepsilon (x_\varepsilon,\theta_\varepsilon) -\partial_{\theta}u_\varepsilon(x_\varepsilon, \theta_\varepsilon)\partial_{\theta}u_\varepsilon (y, \theta_\varepsilon )]e^{\frac{u_\varepsilon(y, \theta_\varepsilon) - u_\varepsilon(x_\varepsilon, \theta_\varepsilon)}{\varepsilon }} K(x_\varepsilon-y)dy\\
&\geq \frac{1}{\varepsilon} \left[ \int_{\Omega} p_\varepsilon (x_\varepsilon,\theta_\varepsilon) e^{\frac{u_\varepsilon(y, \theta_\varepsilon) - u_\varepsilon(x_\varepsilon, \theta_\varepsilon)}{\varepsilon }} K(x_\varepsilon-y)dy  \right. \\
&- \left. \frac{1}{2} \int_{\Omega} \partial_{\theta}u_\varepsilon^2(x_\varepsilon, \theta_\varepsilon)e^{\frac{u_\varepsilon(y, \theta_\varepsilon) - u_\varepsilon(x_\varepsilon, \theta_\varepsilon)}{\varepsilon }} K(x_\varepsilon-y)dy \right. \\
&\left.- \frac{1}{2} \int_{\Omega} \partial_{\theta}u_\varepsilon^2(y, \theta_\varepsilon)e^{\frac{u_\varepsilon(y, \theta_\varepsilon) - u_\varepsilon(x_\varepsilon, \theta_\varepsilon)}{\varepsilon }} K(x_\varepsilon-y)dy  \right] \\ 
&\geq \frac{1}{2\varepsilon} \left[ \int_{\Omega} p_\varepsilon (x_\varepsilon,\theta_\varepsilon) e^{\frac{u_\varepsilon(y, \theta_\varepsilon) - u_\varepsilon(x_\varepsilon, \theta_\varepsilon)}{\varepsilon }} K(x_\varepsilon-y)dy -  \int_{\Omega} p_\varepsilon(y, \theta_\varepsilon)e^{\frac{u_\varepsilon(y, \theta_\varepsilon) - u_\varepsilon(x_\varepsilon, \theta_\varepsilon)}{\varepsilon }} K(x_\varepsilon-y)dy  \right] \\
&\geq 0.
\end{align*}
We deduce thanks to \eqref{RHS} and \eqref{LHS} (and noticing $\frac{2(\partial_{x \theta} u_\varepsilon)^2(x_\varepsilon, \theta_\varepsilon)}{\varepsilon} \geq 0$) that
\begin{equation*}\label{pPfLip2}
\begin{aligned}
\frac{1}{2\varepsilon} \left[  \frac{ \partial_{xx} u_\varepsilon (x_\varepsilon, \theta_\varepsilon)}{\varepsilon}  + \varepsilon\partial_{\theta\theta} u_\varepsilon (x_\varepsilon, \theta_\varepsilon) \right]^2 &\leq  \frac{1}{\varepsilon}\left[ \left(\frac{ \partial_{xx} u_\varepsilon (x_\varepsilon, \theta_\varepsilon)}{\varepsilon} \right)^2 + (\varepsilon\partial_{\theta\theta} u_\varepsilon (x_\varepsilon, \theta_\varepsilon))^2 \right]\\
&\leq \frac{C \sqrt{p_\varepsilon (x_\varepsilon, \theta_\varepsilon)}}{\varepsilon}.   
\end{aligned}
\end{equation*}
Hence, using the original equation \eqref{PfLip1}, we deduce that 
\begin{equation}
    \left[- p_\varepsilon (x_\varepsilon, \theta_\varepsilon) + \int_{\Omega } (1 - e^{\frac{u_\varepsilon(y, \theta_\varepsilon) - u_\varepsilon(x_\varepsilon, \theta_\varepsilon) }{\varepsilon} } )K(x_\varepsilon - y)dy - R(x_\varepsilon, \theta_\varepsilon) + \rho_\varepsilon (x_\varepsilon) \right]^2 \leq C\sqrt{p_\varepsilon (x_\varepsilon, \theta_\varepsilon)}.
\end{equation}
Thanks to the $L^\infty$-bounds on $K$, $R$ and $\rho_\varepsilon$ (assumption \eqref{H2Chap5} \eqref{K} and Lemma \ref{bdrho}), it follows that $p_\varepsilon(x_\varepsilon, \theta_\varepsilon)$ is uniformly bounded with respect to $\varepsilon$. The conclusion follows.\\
\textbf{Case 2 :  $(x_\varepsilon, \theta_\varepsilon) \in \partial \Omega \times ]-A,A[.$  } First remark that in this case, $p_\varepsilon (x_\varepsilon , \theta_\varepsilon ) = |\partial_\theta u_\varepsilon (x_\varepsilon, \theta_\varepsilon) |^2$. We claim that $p_\varepsilon$ verifies also the Neumann boundary conditions at $(x_\varepsilon, \theta_\varepsilon)$. Indeed, according to the Neumann boundary conditions satisfied by $u_\varepsilon$, we can use a reflective argument and differentiate $p_\varepsilon$ on the boundary. We obtain  
\[\partial_x p_\varepsilon (x_\varepsilon, \theta_\varepsilon)  =  \frac{2 \partial_x u_\varepsilon (x_\varepsilon, \theta_\varepsilon) \partial_{xx} u_\varepsilon (x_\varepsilon, \theta_\varepsilon)}{\varepsilon^2} +  2 \partial_\theta u_\varepsilon (x_\varepsilon, \theta_\varepsilon) \partial_{x \theta} u_\varepsilon (x_\varepsilon, \theta_\varepsilon)  = 0 \]
because 
\[ \partial_x u_\varepsilon (x_\varepsilon, \theta_\varepsilon) = 0  \quad \text{ and } \quad  \partial_{x\theta } u_\varepsilon ( x_\varepsilon, \theta_\varepsilon ) = 0.\]
Since $p(x_\varepsilon, \theta_\varepsilon) = \max p_\varepsilon$, we deduce that 
\[ -\partial_{xx} p_\varepsilon (x_\varepsilon, \theta_\varepsilon) \geq 0.\]
We conclude that \eqref{LHS} and \eqref{RHS} hold also true in this case and the conclusion follows from the same computations as in the previous case. \\
\textbf{Case 3 :  $(x_\varepsilon, \theta_\varepsilon) \in \Omega \times \left\lbrace \pm A \right\rbrace$. } This case is treated in the same manner as the previous case. 
\end{proof}

\subsection{The bounds on $\rho_\varepsilon$}

We recall the equation \eqref{eqrho} satisfied by $\rho_\varepsilon$:
\begin{equation*}\tag{$E_\rho$}
\left\lbrace
\begin{aligned}
&-\partial_{xx} \rho_\varepsilon + L \rho_\varepsilon = \int_{-A}^A R(x,\theta) n_\varepsilon (x, \theta) d\theta - \rho_\varepsilon^2 &&\quad \text{ in } \Omega, \\
&\partial_{\nu_x} \rho_\varepsilon = 0 && \quad \text{ on } \partial \Omega.
\end{aligned}
\right.
\end{equation*}

\begin{proof}[Proof of \textit{3.} of Theorem \ref{Regularitythm}]
The uniform bound from above and the $W^{2,p}$ bounds on $\rho_\varepsilon$ are already provided in   Lemma \ref{bdrho}. Here we prove the uniform lower bound.
We start by proving that $0 <c \leq \sup \rho_\varepsilon $. Next, we prove that $c < \rho_\varepsilon$ holds true in the whole domain $\Omega$. \\
\textbf{A lower bound on $\sup \rho_\varepsilon$. } 
Assume by contradiction that there exists a sequence $\varepsilon_k$ such that 
\[ \varepsilon_k \underset{ k \rightarrow +\infty}{\longrightarrow } 0 \quad  \text{ and } \quad  \sup\rho_{\varepsilon_k} \underset{ k \rightarrow +\infty}{\longrightarrow } 0 . \]
Next, if we multiply \eqref{E} by $\xi_{\varepsilon_k}$ (introduced in \eqref{vpchap52}) and we integrate by part, we obtain 
\[ \mu_\varepsilon \int_{-A}^A \int_{\Omega} n_{\varepsilon_k} \xi_{\varepsilon_k} dx d\theta = - \int_{-A}^A \int_{\Omega}  \rho_{\varepsilon_k} n_{\varepsilon_k} \xi_{\varepsilon_k} dx d\theta .\]
We deduce thanks to \eqref{H5} that for $k$ large enough, it holds 
\[ \frac{|\lambda(\theta_0,0)|}{2} \leq -\mu_{\varepsilon_k} \leq \sup \rho_{\varepsilon_k} \frac{\int_{-A}^A \int_{\Omega} n_{\varepsilon_k} \xi_{\varepsilon_k} dx d\theta }{\int_{-A}^A \int_{\Omega} n_{\varepsilon_k} \xi_{\varepsilon_k} dx d\theta } .\]
It is in contradiction with the hypothesis $\sup \rho_{\varepsilon_k} \underset{k \rightarrow  + \infty}{\longrightarrow} 0$. Therefore, there exists a constant $c>0$ such that 
\begin{equation}\label{bdrho2}
    \forall \varepsilon \in ]0, \varepsilon_0[, \qquad c \leq \sup \rho_\varepsilon .
\end{equation}

\vspace{0.5cm}

\noindent \textbf{The lower bound on $\rho_\varepsilon$ in the whole domain $\Omega$. } 
 Let $\varepsilon < \varepsilon_0$ and $x_0 \in \Omega$ be such that
\[\rho_\varepsilon(x_0) = \sup \rho_\varepsilon.\]
We conclude thanks to \eqref{bdrho2} and the Lipschitz estimates obtained in the second item of Theorem \ref{Regularitythm} that for all $x \in \Omega$
\[\rho_\varepsilon (x)  =\int_{-A}^{A} e^\frac{u_\varepsilon (x, \theta)}{\varepsilon} d\theta  = \int_{-A}^{A} e^\frac{u_\varepsilon (x, \theta) - u_\varepsilon (x_0, \theta) + u_\varepsilon (x_0, \theta)}{\varepsilon} d\theta \geq  \rho_\varepsilon(x_0) e^{-C} \geq ce^{-C} .\]

\end{proof}

\subsection{The bounds on $u_\varepsilon$}

\begin{proof}[Proof of 4. of Theorem \ref{Regularitythm}]
First, we prove that there exists $a>0$ such that $-a < u_\varepsilon$.  Thanks to the third item of Theorem \ref{Regularitythm}, we know that there exists $c>0$ such that for all $\varepsilon$ small enough we have
\[ c < \int_{-A}^A n_\varepsilon(x,\theta) d \theta .\]
We deduce the existence of $(x_0, \theta_0) \in \Omega \times ]-A,A[$ such that 
\[\frac{c}{2A} \leq n_\varepsilon(x_0, \theta_0).\]
Hence, it follows 
\[ \varepsilon \log \left( \frac{c}{2A} \right) \leq u_\varepsilon (x_0, \theta_0).\]
We conclude thanks to the Lipschitz estimates established in the second item of Theorem \ref{Regularitythm} that
\begin{equation}\label{bd2}
    \forall (x, \theta ) \in \Omega \times ]-A,A[, \quad -a \leq -2CA + \varepsilon  \left[ \log\left(\frac{c}{2A}\right) - C |\Omega|\right]   \leq u_\varepsilon (x, \theta).
\end{equation}

\vspace{0.5cm}

Next, we prove that $\underset{\varepsilon \rightarrow 0 }{\lim} \ \underset{ (x,\theta) \in \Omega \times ]-A,A[}{\sup} u_\varepsilon(x,\theta) \leq 0 $.\\
We prove it by contradiction. Assume that there exists $a>0$ and sequences $\varepsilon_k, (x_k, \theta_k)$ such that 
\[ \varepsilon_k \underset{k\rightarrow +\infty}{\longrightarrow} 0 \qquad \text{ and } \qquad  u_{\varepsilon_k}(x_k, \theta_k)>a.\]
Using the Lipschitz estimates provided by the second item of Theorem \ref{Regularitythm}, it follows for all $ \theta \in  \left( B_{\frac{a}{4 C}}(\theta_k) \cap  ]-A,A[\right)$
\[u_{\varepsilon_k}(x_k, \theta) = u_{\varepsilon_k}(x_k, \theta)- u_{\varepsilon_k}(x_k, \theta_k) + u_{\varepsilon_k}(x_k, \theta_k) \geq - C|\theta - \theta_k| + a \geq \frac{a}{2}\]
where $C$ corresponds to the Lipschitz estimate given by \eqref{Lipeq}. We deduce that 
\[ \rho_{\varepsilon_k} (x_k) \geq \min  \left( 2A,\frac{a}{4C} \right) e^\frac{a}{2\varepsilon_k}.\]
We conclude that $\underset{ k \rightarrow +\infty}{\liminf} \rho_{\varepsilon_k}(x_k) = +\infty$. This is in contradiction with the $L^\infty$ bounds on $\rho_\varepsilon$ established in the third item of Theorem \ref{Regularitythm}.

\end{proof}

\section{Convergence to the Hamilton-Jacobi equation}\label{convergence}

\begin{proof}[Proof of Theorem \ref{main}]
We prove here the 3 items of Theorem \ref{main}. 

\textbf{\textit{Proof of 1. Convergence of $\rho_\e$.}}  Thanks to the third item of Theorem \ref{Regularitythm}, it follows that for $\varepsilon$ small enough $0 < c \leq \rho_\varepsilon \leq C$ and $\|\rho_\varepsilon \|_{W^{2,p}(\Omega)} \leq C$. We deduce from the classical Sobolev injection (see \cite{brezis}) that $\rho_\varepsilon$ converges, along subsequences, strongly in $W^{1,p}(\Omega)$ and in particular uniformly to $\rho$ and $\rho$ verifies 
\[ 0 < c \leq \rho \leq C.\]

\vspace{0.5cm}

\textbf{\textit{Proof of 2. } } 
(i) {\bf Convergence to the Hamilton-Jacobi equation.} The convergence of $u_\e$ to a viscosity solution of  the Hamilton-Jacobi equation \eqref{HJChap5} can be obtained thanks to the regularity results given in Theorem \ref{Regularitythm} and  a perturbed test function argument,   following the heuristic argument provided in Section \ref{subsec:result} .

From the Lipschitz estimates and the bounds established in the second and the fourth items of Theorem \ref{Regularitythm}, we deduce thanks to the Arzela-Ascoli Theorem that up to a subsequence, $(u_\varepsilon)_{\varepsilon>0}$ converges locally uniformly to some continuous function $u$. Moreover, the limit function $u$ does not depend on $x$.\\
We prove that $u$ is a viscosity solution of 
\[ \left\lbrace
\begin{aligned}
&- |\partial_\theta u |^2 =- \lambda(\theta, \rho), \\
& \partial_{\nu_\theta} u (\pm A) = 0,
\end{aligned}
\right. \]
with $\lambda(\theta, \rho)$ the principal eigenvalue of \eqref{vpchap5}. First, we focus on the equation in the interior of the domain and then we treat the boundary conditions.

\vspace{0.4cm}

\textbf{The interior equation. } We recall that for a fixed value $\theta$, Proposition \ref{propev} provides the existence of a sequence of principal eigenvalues $\lambda(\theta, \rho_\varepsilon )$ associated with a sequence of positive eigenfunctions $(\psi_\varepsilon^\theta)_{\varepsilon>0}$ of the operator $-\partial_{xx} + L - (R(x, \theta) - \rho_\varepsilon)$ with Neumann boundary conditions:
\begin{equation}\label{vpchap53}
    \text{i.e.  }\left\lbrace
    \begin{aligned}
   & -\partial_{xx} \psi_\varepsilon^\theta  + L(\psi_\varepsilon^\theta) - (R(x, \theta) - \rho_\varepsilon) \psi_\varepsilon^\theta = \lambda (\theta, \rho_\varepsilon) \psi_\varepsilon^\theta \quad &&\text{ in } \Omega,\\
    &\partial_{\nu_x} \psi_\varepsilon^\theta = 0 &&\text{ on } \partial \Omega.    
    \end{aligned}
    \right.
\end{equation}
Since $\psi_\varepsilon^\theta>0$, we introduce
\[ \Psi_\varepsilon^\theta = \ln ( \psi_\varepsilon^\theta) .\]
Let  $\phi$ be a test function such that   $u - \phi $ has a strict maximum at  $\theta \in ]-A,A[$. Then, there exists $(x_\varepsilon, \theta_\varepsilon) \in \overline{\Omega} \times ]-A,A[$ such that 
\[ \theta_\varepsilon \underset{ \varepsilon \rightarrow 0}{\longrightarrow} \theta \quad \text{ and } \underset{(x, \theta) \in \overline{\Omega } \times ]-A,A[}{\max} u_\varepsilon (x, \theta) - \phi(\theta) - \varepsilon \Psi^{\theta_\varepsilon}_\varepsilon (x)=  u_\varepsilon (x_\varepsilon, \theta_\varepsilon) - \phi(\theta_\varepsilon) - \varepsilon \Psi^{\theta_\varepsilon}_\varepsilon (x_\varepsilon) .\]
We distinguish two cases: either $x_\varepsilon \in \Omega$ or $x_\varepsilon \in \partial \Omega$. \\
\textbf{Case 1: $x_\varepsilon \in \Omega$. } Since $u_\varepsilon$ is a classical solution of \eqref{EHCchap5}, we deduce that it is also a viscosity solution, therefore 
\begin{align*}
&-\frac{ \partial_{xx} (\phi(\theta_\varepsilon) + \varepsilon \Psi^{\theta_\varepsilon}_\varepsilon (x_\varepsilon))}{\varepsilon}- \frac{[\partial_x (\phi(\theta_\varepsilon) + \varepsilon \Psi^{\theta_\varepsilon}_\varepsilon (x_\varepsilon))]^2}{\varepsilon^2} + \int_{\Omega} [ 1 - e^{ \Psi_\varepsilon^{\theta_\varepsilon} (y) - \Psi_\varepsilon^{\theta_\varepsilon} (x_\varepsilon)}]K(x_\varepsilon -y)dy \\
&- \varepsilon \partial_{\theta\theta} (\phi(\theta_\varepsilon) + \varepsilon \Psi^{\theta_\varepsilon}_\varepsilon (x_\varepsilon) ) - [\partial_\theta (\phi(\theta_\varepsilon) + \varepsilon \Psi^{\theta_\varepsilon}_\varepsilon (x_\varepsilon))]^2 - R(x_\varepsilon , \theta_\varepsilon ) + \rho_\varepsilon (x_\varepsilon)\leq 0. 
\end{align*}
Remarking that $\phi$ does not depend on $x$ and the $\theta$ value is fixed in $\Psi^{\theta_\varepsilon}_\varepsilon$, we deduce that 
\begin{equation}
 \begin{aligned}
    &- \partial_{xx} \Psi^{\theta_\varepsilon}_\varepsilon (x_\varepsilon ) - [ \partial_x \Psi^{\theta_\varepsilon}_\varepsilon (x_\varepsilon) ]^2 + \int_{\Omega} [ 1 - e^{ \Psi_\varepsilon^{\theta_\varepsilon} (y) - \Psi_\varepsilon^{\theta_\varepsilon} (x_\varepsilon)}]K(x_\varepsilon -y)dy - R(x_\varepsilon , \theta_\varepsilon) + \rho_\varepsilon(x_\varepsilon)  \\
    &- \varepsilon \partial_{\theta \theta } \phi (\theta_\varepsilon ) - [\partial_\theta \phi (\theta_\varepsilon) ]^2 \leq 0.
\end{aligned}
\end{equation}
Next, we observe that \eqref{vpchap53} implies 
\[ - \partial_{xx} \Psi^{\theta_\varepsilon}_\varepsilon (x_\varepsilon ) - [ \partial_x \Psi^{\theta_\varepsilon}_\varepsilon (x_\varepsilon) ]^2 + \int_{\Omega} [ 1 - e^{ \Psi_\varepsilon^{\theta_\varepsilon} (y) - \Psi_\varepsilon^{\theta_\varepsilon} (x_\varepsilon)}]K(x_\varepsilon -y)dy - R(x_\varepsilon , \theta_\varepsilon) + \rho_\varepsilon (x_\varepsilon)  = \lambda(\theta_\varepsilon, \rho_\varepsilon ) .\]
Therefore, passing to the limit $\varepsilon \rightarrow 0$, thanks to the continuity of $ \lambda(\theta, \rho )$ with respect to $\theta$ and $\rho$ (Proposition \ref{propev}), it follows that
\[ - [ \partial_\theta \phi (\theta) ]^2 \leq - \lambda(\theta, \rho).\]
\textbf{Case 2 : $x_\varepsilon \in \partial \Omega$. } First, we remark that in this case, 
\[ -\partial_x u_\varepsilon (x_\varepsilon, \theta_\varepsilon) = -\partial_x \Psi_\varepsilon^{\theta_\varepsilon}(x_\varepsilon) = 0.\]
Therefore, we deduce that 
\[ -\partial_x [u_\varepsilon(x_\varepsilon, \theta_\varepsilon) - \phi(\theta_\varepsilon) - \varepsilon \Psi_\varepsilon^{\theta_\varepsilon} (x_\varepsilon) ] = 0.\]
Moreover, since $(u_\varepsilon - \phi - \varepsilon \Psi_\varepsilon^{\theta_\varepsilon})(x_\varepsilon, \theta_\varepsilon) = \max (u_\varepsilon - \phi - \varepsilon \Psi_\varepsilon^{\theta_\varepsilon})$, we have firstly by a reflective argument that 
\begin{equation}\label{incase2}
-\partial_{xx} (u_\varepsilon - \phi - \varepsilon \Psi_\varepsilon^{\theta_\varepsilon}) (x_\varepsilon , \theta_\varepsilon ) \geq 0, 
\end{equation}
and secondly, we have
\begin{equation}\label{incase22}
u_\varepsilon(y, \theta_\varepsilon ) - u_\varepsilon (x_\varepsilon , \theta_\varepsilon) \leq \varepsilon[ \Psi_\varepsilon (y)  - \Psi_\varepsilon (x_\varepsilon)].
\end{equation}
The inequalities \eqref{incase2} and \eqref{incase22} lead to
\begin{align*}
& - \partial_{xx} \varepsilon\Psi_\varepsilon^{\theta_\varepsilon}(x_\varepsilon) \leq -\partial_{xx} u_\varepsilon (x_\varepsilon, \theta_\varepsilon) \\
\text{ and } \quad &  \int_{\Omega} [1-e^{\Psi_\varepsilon^{\theta_\varepsilon} (y) - \Psi_\varepsilon^{\theta_\varepsilon} (x_\varepsilon)}]K(x_\varepsilon - y ) dy \leq \int_{\Omega} [1-e^{\frac{u_\varepsilon (y , \theta_\varepsilon) - u_\varepsilon (x_\varepsilon, \theta_\varepsilon)}{\varepsilon}}]K(x_\varepsilon - y ) dy .
\end{align*}
Therefore, the conclusion follows from similar computation as above.  

\vspace{0.4cm}

\textbf{The boundary conditions. } Let $\phi$ be a test function such that $u-\phi$ has a strict maximum at $A$ (the proof works the same for $-A$). Then, there exists $(x_\varepsilon , \theta_\varepsilon) \in \overline{\Omega} \times [-A, A ]$ such that 
\[   \theta_\varepsilon \underset{ \varepsilon \to 0}{\rightarrow }  A \quad \text{ and } \quad \underset{ (x, \theta ) \in  \overline{\Omega} \times [-A, A ]}{\max} u_\varepsilon - \phi = (u_\varepsilon - \phi)(x_\varepsilon , \theta_\varepsilon).    \]
We distinguish two cases: $(x_\varepsilon, \theta_\varepsilon) \in \overline{\Omega} \times ]-A, A[$ or $(x_\varepsilon, \theta_\varepsilon) \in \overline{\Omega} \times \left\lbrace  A \right\rbrace$.
\begin{description}
\item [Case 1 :]$(x_\varepsilon, \theta_\varepsilon) \in \overline{\Omega} \times ]-A, A[$. In this case, by a similar analysis as above, we deduce that 
\[ -[\partial_\theta \phi(\theta_\varepsilon)]^2 \leq -\lambda(\theta_\varepsilon, \rho) + o_\varepsilon(1).\]
\item [Case 2 :]$(x_\varepsilon, \theta_\varepsilon) \in \overline{\Omega }\times \left\lbrace A \right\rbrace$. In this case, since the maximum is reached on the boundary, we deduce thanks to the boundary conditions of $u_\varepsilon$
\[-\partial_{\nu_\theta} \phi(\theta_\varepsilon) =  \partial_{\nu_\theta} (u_\varepsilon - \phi) (x_\varepsilon, \theta_\varepsilon) \geq 0.\]
\end{description}
Taking the inferior limit, we conclude that
\[ \min ( -[\partial_\theta \phi(A)]^2 + \lambda(A, \rho) , \partial_{\nu_\theta } \phi (\pm A) )\leq  0,  \] 
which corresponds to the boundary condition in the viscosity sense.

Finally, $u$ is a sub-solution of \eqref{HJChap5} in a viscosity sense. With similar arguments, $u$ is also a super-solution. We conclude that $u$ is a viscosity solution of \eqref{HJChap5}.

(ii) {\bf The constraint.} The constraint $\max_\theta u(\theta)=0$ in \eqref{HJChap5} is a consequence of the Hopf-Cole transformation \eqref{H-C} and the fact that $\rho_\e$ remains bounded   away from $0$, uniformly in $\e$. 

If the constraint does not hold true, it follows, thanks to \eqref{boundueps}, that $\underset{\theta \in [-A,A]}{\sup}u(\theta)< -a < 0$.  Hence for $\varepsilon$ small enough and thanks to the uniform convergence of $u_\e$ to $u$, we deduce that $\underset{(x,\theta) \in \Omega \times [-A,A]}{ \max}u_\varepsilon(x,\theta)< -\frac{a}{2} $, which implies that $\rho_\varepsilon < c$ for $ \varepsilon$ sufficiently small. This is in contradiction with the third item of Theorem \ref{Regularitythm}. We conclude that $\underset{\theta \in [-A,A]}{ \max} u(\theta) = 0$. 

 \vspace{0.5cm}
 
\textbf{\textit{Proof of 3. The convergence of $n_\e$ and the inclusion property.}} The first inclusion property in  \eqref{inclusion}   can be obtained thanks to the Hopf-Cole transformation \eqref{H-C} and the uniform convergence of $u_\e$ to $u$. The second inclusion property in  \eqref{inclusion} is a consequence of the Hamilton-Jacobi equation \eqref{HJChap5}  and the fact that the zero level set of $u$ is also the set of the maximum points of $u$. We detail these arguments below.

 Thanks to the $L^\infty$ bounds on $\rho_\varepsilon$ (third point of Theorem \ref{Regularitythm}), we deduce that 
\[ c \leq \| n_\varepsilon \|_{L^1(\Omega \times ]-A,A[)} \leq C. \]
It follows that $n_\varepsilon$ converges up to a subsequence and in the sense of measures to a measure $n$. The measure $n$ is non-negative and not trivial. We next prove that 
\[ \mathrm{supp}\  n_\varepsilon \subset \Omega \times \left\lbrace \theta \in ]-A,A[ \ | \ u(\theta) = 0 \right\rbrace  .\]
Indeed, let $\phi \in C^\infty_c (\Omega \times ]-A,A[)$ be any positive test function such that
\begin{equation}\label{support}
\mathrm{supp} \ \phi \subset \Omega \times  \left\lbrace \theta \in ]-A,A[ \ | \ u(\theta) = 0 \right\rbrace^c.
\end{equation}
We prove that $\int_{\Omega} \int_{-A}^A \phi(x, \theta ) n(x, \theta) dx d\theta = 0$.\\
To this end, we introduce $-a = \underset{\mathrm{ supp } \ \phi}{\sup} u $. According to \eqref{support}, it follows that $a>0$. We deduce that for all $\varepsilon$ small enough and all $(x,\theta) \in \Omega \times \mathrm{supp } \ \phi$, we have 
\begin{equation}
    u_\varepsilon(x, \theta) \leq -\frac{a}{2}.
\end{equation}
We conclude that 
\begin{align*}
\int_\Omega \int_{-A}^A \phi (x,\theta) n(x,\theta) d\theta dx  &= \int_\Omega \int_{\mathrm{ supp} \ \phi(x, \cdot)} \phi (x,\theta) n(x,\theta) d\theta dx \\
&= \underset{ \varepsilon \rightarrow 0}{\lim} \ \int_\Omega \int_{\mathrm{ supp} \ \phi(x, \cdot)} \phi (x,\theta) n_\varepsilon(x, \theta) d\theta dx \\
&= \underset{ \varepsilon \rightarrow 0}{\lim} \int_\Omega  \int_{\mathrm{ supp} \ \phi(x, \cdot)} \phi (x,\theta) e^{\frac{u_\varepsilon(x, \theta)}{\varepsilon}} d\theta  dx  \\
&\leq \underset{ \varepsilon \rightarrow 0}{\lim} \int_\Omega  \int_{\mathrm{ supp} \ \phi(x, \cdot)} \phi (x,\theta) e^{\frac{-a}{2\varepsilon}} d\theta  dx  \\
&= 0.
\end{align*}
We finally prove that 
\begin{equation}
\label{eq:sub0R}
\left\lbrace u(\theta) = 0 \right\rbrace \subset \left\lbrace \lambda(\theta,\rho)=0\right\rbrace.
\end{equation}
To this end, note first that since $u$ is a Lipschitz continuous function, it is a.e. differentiable. Therefore, \eqref{HJChap5} implies that
$$
\lambda(\theta,\rho)\geq 0, \quad \text{for a.e. $\theta$.}
$$
Moreover, since $\lambda$ is continuous with respect to $\theta$ the above inequality holds indeed for all $\theta$. To prove \eqref{eq:sub0R}, it is therefore enough to prove that for any $\theta_0$ such that $u(\theta_0)=0$, we have
$$
\lambda(\theta_0,\rho)\leq 0.
$$
This property can be derived by testing the equation in  \eqref{HJChap5} against the test function $\varphi(\theta)\equiv 0$ at the point $\theta_0$ for a viscosity  subsolution criterion.

This concludes the proof of \textit{3.} 
\end{proof}

\section*{Appendix A- Existence and properties of $\lambda(\theta, \rho)$}\label{AppA}

\addcontentsline{toc}{section}{Appendix A- Existence and properties of $\lambda(\theta, \rho)$}

In this section, we first establish a Hopf Lemma. It is obtained by a classical argument but for the sake of completeness and because of the presence of the less classical non-local operator $L$, we provide the proof. Next, we verify the existence of $\lambda(\theta, \rho)$. To finish, we provide the proof of some properties of $\lambda$ already stated in the article (namely Propositions \ref{derivlambda} and \ref{propo:lambda:result2}). 

\subsection*{A.1- A Hopf Lemma}

In this section we prove the following Hopf Lemma
\begin{lemma}[Hopf Lemma]\label{Hopf}
Let $u$ be a smooth function defined on $\Omega$ such that 
\begin{equation}\label{HypHopf}
-\partial_{xx} u + L(u) + c(x) u \geq 0 ,
\end{equation}
with $c$ a non-negative bounded smooth function. If there exists $x_0 \in \partial \Omega$ such that $\underset{x \in \Omega}{\min} \ u(x) = u(x_0) < 0$ then either $u$ is constant or 
\begin{equation}
     \partial_{\nu_x} u(x) < 0.
\end{equation} 
\end{lemma}

The proof is in the spirit of the classical proof of the Hopf Lemma (see \cite{Evans-PDE} p.250).

\begin{proof}
Up to a scaling, there is no loss of generality if we assume that $B(0,1) \subset \Omega$ and $x_0=1$. 
Next, we define 
\[v(x) = \left[ e^{-\frac{3}{4}\lambda } - e^{-\lambda \max (0,|x|^2 - \frac{1}{4})} \right] 1_{B(0,1)}(x),\]
for $\lambda$ a positive constant. We underline that $v(x) = e^{-\frac{3}{4}\lambda}-1$ in $B(0,\frac{1}{2})$. Next, we claim that by taking $\lambda$ large enough, for all $x \in B(0,1) \backslash B(0, \frac{3}{4})$ there holds
\begin{equation}\label{HopfPf1}
\begin{aligned}
&- \partial_{xx} v (x) = 2 \lambda e^{-\lambda \max(0,|x|^2-\frac{1}{4}) } (2 \lambda |x|^2 - 1) > 0 \\
\text{ and }\quad & L v(x) \geq \int_{B(0,1)} e^{ - \lambda \max (0,|y|^2 - \frac{1}{4}) }K(x-y)dy - \int_\Omega e^{- \lambda\max (0,|x|^2 - \frac{1}{4})}  K(x-y)dy > 0.
\end{aligned}
\end{equation}
The first inequality of \eqref{HopfPf1} follows from a straightforward computation. For the second inequality, according to the assumption \eqref{K}, we have 
\[\underset{\lambda \rightarrow +\infty}{\liminf} \int_{B(0,1)} e^{ - \lambda \max (0,|y|^2 - \frac{1}{4}) }K(x-y)dy - \int_\Omega e^{- \lambda\max (0,|x|^2 - \frac{1}{4})}  K(x-y)dy \geq c_K |B(0, \frac{1}{2})| > 0.\]
Therefore, if $\lambda$ is large enough, \eqref{HopfPf1} holds true.\\

Next, we claim that if $u$ is not constant, the minimum can not be reached in the interior of $\Omega$. Otherwise,  we deduce the existence of $x_1 \in \Omega$ such that $u(x_1) = \underset{x \in \Omega}{\min } \ u <0 $. Since $c$ is  nonnegative, we have
\[-\partial_{xx} u (x_1) \leq 0, \quad Lu(x_1) < 0 \quad \text{ and } \quad c(x_1)u(x_1) \leq 0.\]
Therefore, we deduce that 
\[ -\partial_{xx} u (x_1)  + Lu(x_1)  + c(x_1)u(x_1) <0 .\]
This is in contradiction with the assumption \eqref{HypHopf}. \\
We deduce that $\underset{ x \in \partial B(0, \frac{3}{4})}{\min} \ u(x_0) - u(x) <0$. Next, taking $\varepsilon$ small enough, there holds that 
\[ \forall x \in \partial B(0, \frac{3}{4}), \quad u(x_0) - u(x) - \varepsilon v(x) < 0.\]
Since $v=0$ on $\partial B(0,1)$ and by definition of $x_0$, it follows
\[\forall x \in \partial B(0, 1), \quad u(x_0) - u(x) - \varepsilon v (x) \leq 0.\]
Moreover, according to \eqref{HypHopf}, \eqref{HopfPf1} and remarking that $u(x_0)-\varepsilon v(x) \geq 0$ for $\varepsilon$ small enough, we deduce that for all $x \in B(0,1) \backslash B(0, \frac{3}{4})$  
\[ -\partial_{xx} (u(x_0) - u(x) - \varepsilon v(x))+L(u(x_0) - u - \varepsilon v)(x)+ c(x) (u(x_0) - u(x) -\varepsilon  v(x)) \leq 0.\]
We deduce thanks to the maximum principle that $u(x_0) - u(x) - \varepsilon  v(x) \leq 0$, for all $x \in B(0,1) \backslash B(0, \frac{3}{4})$. We conclude that 
\[\partial_{\nu_x} u (x_0) \leq  -\partial_{\nu_x} \varepsilon v (x_0) = -\varepsilon 2\lambda e^{-\frac{3\lambda}{4}} < 0 .\]
\end{proof}

\subsection*{A.2- Existence of a principal eigenpair}

\begin{proposition}\label{propev}
Under the hypothesis \eqref{H1Chap5}--\eqref{K}, for a fixed bounded smooth function $\rho$ and a fixed value $\theta \in ]-A,A[$ there exists a principal eigenvalue $\lambda(\theta, \rho)$ of the operator $-\partial_{xx} \psi + L(\psi) - (R(\cdot, \theta) - \rho)\psi$ with Neumann boundary conditions
\begin{equation} \label{EV}
    \text{ i.e. } \left\lbrace 
    \begin{aligned}
    &-\partial_{xx} \psi^{\theta} + L(\psi^{\theta}) - (R (\cdot, \theta) - \rho)\psi^{\theta} = \lambda(\theta, \rho) \psi^{\theta} && \text{ in } \Omega,\\
    &\partial_{\nu_x} \psi = 0 \text{ on } \partial \Omega.
    \end{aligned}
    \right.
\end{equation}
The associated eigenfunction $\psi^{\theta}$ has a constant sign and is unique up to multiplication by a constant. Moreover, the function $\lambda(\theta, \rho)$ and $\psi^{\theta}$ are $C^1$ with respect to $\theta$ and $\rho \in H^1(\Omega)$.
\end{proposition}

In the following, we will consider that $\psi^{\theta}$ is positive and of $L^2$ norm equal to $1$.

\begin{proof}
First, we prove the existence of the principal eigenpair by verifying that we can apply the Krein Rutman Theorem (see \cite{Smoller} p 122). Since it is classical, we do not provide all the details. The cone of functions where we apply the Krein-Rutmann Theorem is 
\[K = \overline{ \left\lbrace u \in C^{1+\alpha}(\Omega) \ | \ u > 0 \ \text{ and } \ \partial_{\nu_x} u = 0 \right\rbrace}.\]
We define $\mathcal{L}(v)$ as the unique solution of 
\begin{equation*}
    \left\lbrace
    \begin{aligned}
    &-\partial_{xx} \mathcal{L}(v) + \int_{\Omega} [\mathcal{L}(v)(x) - \mathcal{L}(v)(y)]K(x-y)dy - (R(\cdot, \theta) - \rho - C) \mathcal{L}(v) = v \text{ in } \Omega, \\
    &\partial_{\nu_x} \mathcal{L}(v) = 0 \text{ on } \partial \Omega
    \end{aligned}
    \right.
\end{equation*}
where $C > \underset{x \in \Omega}{\sup} (R(x, \theta) - \rho(x))$ and $v \in K$. The operator $\mathcal{L}$ is linear, compact thanks to the elliptic estimates. We have to prove that 
\[ \forall v \in K \backslash \left\lbrace 0 \right\rbrace, \quad \mathcal{L}(v) \in \mathrm{int}(K).\]
Let $v$ be in $K$ with $v$ not trivial. By elliptic regularity, it follows $\mathcal{L}(v) \in C^{1+\alpha}$ and $\partial_{\nu_x} \mathcal{L}(v) = 0$. It remains to prove that $\mathcal{L}(v) > 0$.\\
First we prove that if $\mathcal{L}(v)$ is constant then it is necessarily a positive constant. Next we prove that if $\mathcal{L}(v)$ varies then $\mathcal{L}(v) >0$.\\
Assume that $\mathcal{L}(v) = c$. Let $\overline{x}\in \Omega$ be such that $v(\overline{x}) > 0$. Moreover, the choice of $C$ gives $-(R(\overline{x}, \theta) - \rho(\overline{x}) -C) >0$ and since $-\partial_{xx} \mathcal{L}(v)  = L(\mathcal{L}(v)) = 0$, we deduce that 
\[\mathcal{L}(v)(\overline{x}) = c = \frac{ v(\overline{x} )}{- (R(\overline{x}, \theta) - \rho(\overline{x}) - C) } >0.\]
Next, we suppose that $\mathcal{L}(v)$ is not constant. Assume by contradiction that there exists $x$ such that $\mathcal{L}(v)(x) \leq 0$. Let $\overline{x}'  \in \overline{\Omega}$ be such that 
\[ \underset{ x \in \Omega}{\inf} \ \mathcal{L}(v)(x)  = \mathcal{L}(v)(\overline{x}').\]
Then either $\overline{x}' \in \Omega$ or $\overline{x}' \in \partial \Omega$. In the first case, we deduce that 
\[ -\partial_{xx}  \mathcal{L}(v) (\overline{x}') \leq 0 \quad \text{ and } \quad L(\mathcal{L}(v) )(\overline{x}') <0,\]
which leads to the following contradiction
\[ 0 \leq v(\overline{x}') = -\partial_{xx}  \mathcal{L}(v) (\overline{x}') + L(\mathcal{L}(v) )(\overline{x}') - (R(\overline{x}', \theta) - \rho(\overline{x}') - C) \mathcal{L}(v) (\overline{x}') < 0.\]
If $\overline{x}' \in \partial \Omega$, since $\mathcal{L}(v)$ is not constant, we deduce from Lemma \ref{Hopf} that $\partial_{\nu_x} \mathcal{L}(v)(\overline{x}')<0$. It is in contradiction with the Neumann boundary condition. 

We conclude that we can apply the Krein Rutman theorem and the conclusion follows.

\bigbreak

Next, we focus on the regularity of $\lambda$ and $\psi^{\theta}$ with respect to $\theta$ and $\rho$. The result follows directly from the implicit function theorem applied to 
\[ G:(\phi, \lambda, \theta, \rho) \in H^1(\Omega) \times \mathbb{R}  \times \mathbb{R} \times H^1(\Omega) \mapsto  \left( -\partial_{xx} \phi + L \phi - [R(\cdot, \theta) - \rho + \lambda] \phi, \int_\Omega \phi(x)^2dx - 1 \right). \]
The interested reader may refer to Theorem 2 of Chapter 11 of \cite{Evans-PDE} for technical details in a finite dimensional setting.

\end{proof}

The existence of the solution of \eqref{vpchap52} is also due to the Krein-Rutman Theorem, therefore we do not provide the proof of existence. 

\subsection*{A.3- Some properties of $\lambda$}

We prove in this subsection two propositions that are stated in section \ref{sec:qual}: Propositions \ref{derivlambda} and \ref{propo:lambda:result2}. 

\begin{proof}[Proof of Proposition \ref{derivlambda}]
First recall from Proposition \ref{propev} that $\lambda$ and $\psi^\theta$ are $C^1$ functions with respect to $\theta$.
Differentiating \eqref{vpchap5} with respect to $\theta$ leads to
\begin{equation}\label{pfderivvp1}
\partial_\theta \lambda(\theta, \rho )  \psi^\theta +  \lambda(\theta, \rho) \partial_{\theta} \psi^\theta =  -\partial_{xx} \partial_{\theta} \psi^\theta + L (\partial_\theta \psi^\theta) - [R(\cdot, \theta) - \rho] \partial_\theta \psi^\theta - \partial_{\theta } R (\cdot, \theta) \psi^{\theta}.
\end{equation}
We multiply \eqref{pfderivvp1} by $\psi^\theta$ and  integrate by parts, recalling $\int_\Omega \psi^{\theta}(x)^2 dx = 1$, to obtain that
\begin{equation}\label{pfderivvp2}
\partial_{\theta } \lambda(\theta, \rho) = - \int_{\Omega}\partial_{\theta }R (x, \theta) \psi^{\theta}(x)^2 dx +  \int_{\Omega} \partial_x \partial_\theta \psi^\theta \partial_x \psi^\theta + L(\psi^\theta) \partial_\theta \psi^\theta - [ R(x, \theta) - \rho - \lambda(\theta, \rho) ] \partial_\theta \psi^\theta \psi^\theta dx . 
\end{equation}
We remark that multiplying \eqref{vpchap5} by $\partial_\theta \psi^\theta$ and integrating by part leads to 
\[ \int_{\Omega}  \partial_x \psi^\theta \partial_x \partial_\theta \psi^\theta + L(\psi^\theta) \partial_\theta \psi^\theta - [ R(x, \theta) - \rho  - \lambda(\theta, \rho) ]  \psi^\theta \partial_\theta \psi^\theta dx  = 0.\]
The conclusion follows. 
\end{proof}


\begin{proof}[Proof of Proposition \ref{propo:lambda:result2}]
We focus on the proof of \eqref{Claim2} since the proof of \eqref{Claim1} follows straightforward computations. 

\bigbreak

Since the function $R$ is $C^2$ with respect to $\theta$ and $g$, according to the implicit function theorem, the eigenfunction $\psi^{\theta}$ is $C^2$ with respect to its parameters $g$ and $\theta$. In this proof, we will take into account in the notations this dependence with respect to the parameters: we will denote $\psi^{\theta}$ by $\psi^{\theta, g}$, $\lambda(\theta)$ by $\lambda(\theta, g)$ and $R$ by $R_g$.

\vspace{0.2cm}

First, we establish that $\int_{\Omega} \left( \partial_\theta \psi^{\theta, g}(x) \right)^2dx \underset{g \to 0}{\longrightarrow} 0$ for a fixed value $\theta \in ]-A,A[$. Since $\| \psi^{\theta,g} \|_{L^2(\Omega)}^2 = 1$, we deduce that $\int_{\Omega} \psi^{\theta,g}(x) \partial_\theta \psi^{\theta,g} (x) dx = 0$. If we differentiate \eqref{EV} with respect to $\theta$, we have that $\partial_\theta \psi^{\theta,g}$ is solution in $\Omega$ of
\[- \partial_{xx} \partial_\theta \psi^{\theta,g} + L(\partial_\theta \psi^{\theta,g}) - [R_g(\cdot , \theta) - \rho] \partial_\theta \psi^{\theta,g} - \partial_\theta R_g(\cdot, \theta) \psi^{\theta,g}  = \lambda(\theta, g) \partial_\theta \psi^{\theta,g} + \partial_\theta \lambda(\theta,g) \psi^{\theta,g}.\]
We then multiply the above equation by $\psi^{\theta, g}$ and  integrate over $\Omega$ to obtain that
\begin{align*}
&\int_{\Omega} \left(\partial_x (\partial_\theta \psi^{\theta,g})\right)^2 dx -
 \frac{\int_{\Omega \times \Omega} [\partial_{\theta} \psi^{\theta,g}(x) - \partial_\theta \psi^{\theta,g}(y)]^2 K(x-y)dydx}{2} - \int_\Omega [R_g(x, \theta) - \rho(x)] \partial_\theta \psi^{\theta,g}(x)^2 dx  \\
&= \lambda(\theta , g) \int_{\Omega}\left(\partial_\theta \psi^{\theta, g} \right)^2dx + \int_{\Omega } \partial_\theta R_g(x, \theta) \psi^{\theta, g}(x) \partial_\theta \psi^{\theta, g}(x) dx + \partial_\theta \lambda(\theta, g) \int_{\Omega} \psi^{\theta, g}(x) \partial_\theta \psi^{\theta, g} (x)dx. 
\end{align*} 
Remarking that $\partial_\theta R_0 = 0$ and recalling that $\int_{\Omega} \psi^{\theta, 0}(x) \partial_\theta \psi^{\theta, 0} (x)dx = 0$, we deduce that $\partial_\theta \psi^{\theta, 0}$ belongs to the eigenspace associated to the principal eigenvalue $\lambda(\theta, 0)$ of the operator $-\partial_{xx} + L - [R_0 - \rho]$. Since this space is one dimensional, engendered by $\psi^{\theta, 0}$ and using again that $\psi^{\theta, 0}$ is orthogonal to $\partial_\theta \psi^{\theta, 0}$, we conclude that 
\[ \int_{\Omega} \left( \partial_\theta \psi^{\theta, g}(x) \right)^2 dx \underset{ g \to 0}{\rightarrow} 0.\]

\vspace{0.2cm}

Next, we prove that this convergence is uniform with respect to $\theta$. By compactness of $[-A, A] \times [0, \widetilde{g}]$ (for some $\widetilde{g}>0$), we deduce that there exists a uniform constant $C>0$ such that for all $(\theta,g) \in [-A, A] \times [0, \widetilde{g}]$ we have
\[ |\partial_{\theta \theta } \psi^{\theta, g} | < C. \]
It follows that for any $\theta_1, \theta_2 \in ]-A, A[$, we have
\[ \int_{\Omega} \left(\partial_\theta \psi^{\theta_1, g}(x) \right)^2dx  \leq \int_{\Omega} \left( \partial_\theta \psi^{\theta_2, g}(x) \right)^2dx   + 2C^2 \mathrm{mes}(\Omega) |\theta_1 - \theta_2|^2\]
(where $\mathrm{mes}(\Omega)$ stands for the Lebesgue mesure of $\Omega$). 
Next, we fix $\mu>0$ and we prove that for $g<g_0$ (for some $g_0>0$) we have (independtly of the choice of $\theta$)
\[ \int_{\Omega} \left( \partial_\theta \psi^{\theta, g}(x) \right)^2 dx < \mu.\]
By compactness of $[-A, A]$, there exists an integer $i_0>0$ and $\theta_i \in ]-A, A[$ with $i \in \left\lbrace 1, ..., i_0 \right\rbrace$ such that 
\[ [-A, A] \subset \underset{ i = 1}{\overset{ i_0 }{\bigcup}} B\left(\theta_i , \sqrt{\frac{ \varepsilon }{ 4 C^2 \mathrm{mes}(\Omega)}} \right). \]
Next, we define $g_i = \sup \left\lbrace g' \in [0, \widetilde{g}] \ : \ \forall g<g', \  \int_{\Omega} \left( \partial_\theta \psi^{\theta_i, g}(x) \right)^2 dx < \frac{\mu}{2} \right\rbrace$ (notice that $g_i>0$). By setting $g_0 = \underset{ i \in \left\lbrace 1, ..., i_0 \right\rbrace}{\min } g_i$, for any $\theta \in ]-A, A[$, we conclude that for all $g<g_0$ we have
\[ \int_{\Omega} \left(\partial_\theta \psi^{\theta, g}(x) \right)^2dx  \leq \int_{\Omega} \left( \partial_\theta \psi^{\theta_i, g}(x) \right)^2dx   + 2C^2 \mathrm{mes}(\Omega) |\theta - \theta_i|^2 < \mu\]
with $i \in \left\lbrace 1, ..., i_0 \right\rbrace$ such that $\theta \in B\left(\theta_i , \sqrt{\frac{ \varepsilon }{ 4 C^2 \mathrm{mes}(\Omega)}} \right)$. It concludes the proof of the uniform convergence.  

\end{proof}

\section*{Appendix B- Proof of Lemma \ref{lemmamu}}\label{AppB}

\addcontentsline{toc}{section}{Appendix B- Proof of Lemma \ref{lemmamu}}

The proof of Lemma \ref{lemmamu} follows essentially the steps of the proof of the convergence of $u_\varepsilon$ (i.e. second item of Theorem \ref{main}). Therefore, we will only emphasize the differences between the two proofs. 
We made the choice to provide the proof of the convergence of $u_\varepsilon$ rather than the convergence of $\mu_\varepsilon$ because it is the result that motivated the current study. 

\begin{proof}[Proof of Lemma \ref{lemmamu}]
We recall the equation satisfied by $\mu_\varepsilon $ and $\xi_\varepsilon$
    \begin{equation*}\tag{\ref{vpchap52}}
\left\lbrace
\begin{aligned}
&-\partial_{xx} \xi_\varepsilon - \varepsilon^2 \partial_{\theta\theta} \xi_\varepsilon + L \xi_\varepsilon  - R \xi_\varepsilon = \mu_\varepsilon \xi_\varepsilon  &&\text{ in } \Omega \times ]-A,A[, \\
& \partial_{\nu_x} \xi_\varepsilon = \partial_{\nu_\theta }\xi_\varepsilon = 0 && \text{ on } \partial ( \Omega \times ]-A,A[).
\end{aligned}
\right.
\end{equation*}
The existence of $\xi_\varepsilon$ is ensured by the Krein-Rutman Theorem. Moreover, according to the Krein-Rutman Theorem, the sign of $\xi_\varepsilon$ is constant. Therefore, we consider that $\xi_\varepsilon > 0$, $\|\xi_\varepsilon\|_{L^2} = 1$ and we define 
\[v_\varepsilon = \varepsilon \ln (\xi_\varepsilon).\]
Next, we prove that $\mu_\varepsilon$ is bounded from below and above respectively by $-\sup R$ and $-\inf R$. \\ First, we focus on the upper bound. Let $(\overline{x}, \overline{\theta}) \in \overline{\Omega} \times [-A,A]$ be such that $\underset{(x, \theta)  \in \overline{\Omega}\times [-A,A]}{\sup} \xi_\varepsilon (x, \theta) = \xi_\varepsilon (\overline{x}, \overline{\theta})$. If $(\overline{x}, \overline{\theta}) \in \Omega \times ]-A,A[$, it follows 
\[ (-\partial_{xx} \xi_\varepsilon - \varepsilon^2 \partial_{\theta\theta} \xi_\varepsilon + L(\xi_\varepsilon))(\overline{x}, \overline{\theta}) \leq 0.\]
From \eqref{vpchap52}, we deduce that 
\[ \mu_\varepsilon \leq - R(\overline{x}, \overline{\theta}) \leq - \inf R.\]
If $(\overline{x}, \overline{\theta})$ belongs to $ \partial \left( \Omega \times ]-A,A[ \right)$, we conclude with a reflective argument and the same computations as in the previous case. In any case, for all $\varepsilon>0$ we have 
\[ \mu_\varepsilon < - \inf{R}.\]
Next, we focus on the lower bound. Let $(\underline{x}, \underline{\theta}) \in \overline{\Omega} \times [-A, A]$ be such that $\underset{ (x, \theta) \in \overline{\Omega} \times [-A,A]}{\inf} \xi_\varepsilon = \xi_\varepsilon (\underline{x}, \underline{\theta})$. With similar arguments as for the upper bound, we deduce that
\[ - \sup \ R \leq -R (\underline{x}, \underline{\theta}) \leq \mu_\varepsilon.\]
Therefore, $\mu_\varepsilon$ is uniformly bounded from below and above thus $\mu_\varepsilon$ converges along subsequences to $\mu$.\\
Next, as we have established Lipschitz and uniform bounds on $u_\varepsilon$, we can prove that there exists a constant $C>0$ such that
\begin{align*}
& |\partial_x v_\varepsilon | < C \varepsilon,  \qquad |\partial_\theta v_\varepsilon | < C, \qquad  -C < \underset{\varepsilon \rightarrow 0}{\lim} \ \underset{  \Omega \times ]-A,A[}{\inf} v_\varepsilon, \qquad \text{ and } \qquad  \underset{\varepsilon \rightarrow 0}{\lim} \ \underset{  \Omega \times ]-A,A[}{\sup} v_\varepsilon\leq 0.
\end{align*}
Therefore, we deduce that $v_\varepsilon$ converges along subsequences to $v$. Moreover, with similar computations as in the proof of the second item of Theorem \ref{main}, we deduce that $v$ is a viscosity solution of 
\begin{equation}\label{HJChap5V}
\left\lbrace
\begin{aligned}
     &-[\partial_\theta v(\theta)]^2 = -\lambda(\theta , -\mu),\\
     &\underset{ \theta \in [-A,A]}{\max} v(\theta) = 0.
     \end{aligned}
     \right.
\end{equation}
Next, we claim that 
\begin{equation}\label{lambdamu}
\lambda(\theta , -\mu) = \lambda( \theta, 0) - \mu.
\end{equation}
We postpone the proof of this claim to the end of this paragraph. Thanks to \eqref{HJChap5V} and \eqref{lambdamu} we deduce that 
\begin{equation*}
\left\lbrace
\begin{aligned}
     &-[\partial_\theta v(\theta)]^2 = -\lambda(\theta , 0) +\mu,\\
     &\underset{ \theta \in [-A,A]}{\max} v(\theta) = 0.
     \end{aligned}
     \right.
\end{equation*}
Remark that $-\lambda(\theta, 0) + \mu \leq 0$ for all $\theta \in [-A,A]$. Next, we introduce $\theta_m \in [-A,A]$ such that 
\[ v(\theta_m) = \underset{ \theta \in [-A,A]}{\max} v(\theta) .\]
It follows that $\partial_\theta v(\theta_m) = 0$ and $-\lambda(\theta_m, 0) + \mu = 0= \max \ ( -\lambda(\theta, 0) + \mu)$. We deduce thanks to \eqref{assumption} that
\[0 =  \max (- \lambda(\theta, 0) + \mu ) = - \min \left(\lambda(\theta, 0)\right) +\mu = -\lambda(\theta_0, 0) + \mu.\]
We conclude that
\[  \lambda(\theta_0, 0) = \mu.\]
We finish the proof by remarking that the previous convergence result holds for any subsequence of $\mu_\varepsilon$. Therefore, we conclude that 
\[\underset{\varepsilon \rightarrow 0}{\lim} \  \mu_\varepsilon = \lambda(\theta_0, 0).\]
\vspace{0.5cm}

It remains to prove \eqref{lambdamu}. Let $\psi_\mu^{\theta}$ be the principal eigenfunction associated to the principal eigenvalue of $\lambda(\theta, -\mu)$ with $\mu$ a constant
    \begin{equation*}
\left\lbrace
\begin{aligned}
&-\partial_{xx} \psi_\mu^{\theta} + L \psi_\mu^{\theta}  - [R(\cdot, \theta) + \mu ] \psi_\mu^{\theta} = \lambda(\theta,-\mu) \psi_\mu^{\theta}  &&\text{ in } \Omega, \\
& \partial_{\nu_x} \psi_\mu^{\theta}  = 0 && \text{ on } \partial \Omega .
\end{aligned}
\right.
\end{equation*}
It follows that
\[ -\partial_{xx} \psi_\mu^{\theta} + L \psi_\mu^{\theta}  - R(\cdot, \theta) \psi_\mu^{\theta} = \left( \lambda(\theta,-\mu) + \mu \right) \psi_\mu^{\theta}.\]
Since $\mu$ is constant, $\psi_\mu^{\theta}>0$ and by the uniqueness of the positive eigenfunction of $-\partial_{xx} + L - R(\cdot, \theta)$ (up to a multiplication by a scalar), we deduce that $\lambda(\theta,-\mu) + \mu  = \lambda(\theta, 0)$.
\end{proof}

\section*{Acknowledgements} Both authors are grateful to Robin Aguil\'ee for fruitful discussions on the biological motivations. S.M. thanks also Denis Roze for  valuable  discussions and references. S.M.  has recieved partial funding from the  ANR project DEEV ANR-20-CE40-0011-01 and the chaire Mod\'elisation Math\'ematique et Biodiversit\'e of V\'eolia Environment - \'Ecole Polytechnique - Museum National d'Histoire Naturelle - Fondation X.

\bibliographystyle{plain}

{\footnotesize

 \bibliography{Biblio}}

\begin{thebibliography}{10}

\bibitem{ABR}
M.~Alfaro, H.~Berestycki, and G.~Raoul.
\newblock The effect of climate shift on a species submitted to dispersion,
  evolution, growth, and nonlocal competition.
\newblock {\em SIAM J. Math. Anal.}, 49(1):562--596, 2017.

\bibitem{ACR}
M.~Alfaro, J.~Coville, and G.~Raoul.
\newblock Travelling waves in a nonlocal reaction-diffusion equation as a model
  for a population structured by a space variable and a phenotypic trait.
\newblock {\em Comm. Partial Differential Equations}, 38(12):2126--2154, 2013.

\bibitem{ADP}
A.~Arnold, L.~Desvillettes, and C.~Pr\'{e}vost.
\newblock Existence of nontrivial steady states for populations structured with
  respect to space and a continuous trait.
\newblock {\em Commun. Pure Appl. Anal.}, 11(1):83--96, 2012.

\bibitem{BMP}
G.~Barles, S.~Mirrahimi, and B.~Perthame.
\newblock Concentration in {L}otka-{V}olterra parabolic or integral equations:
  a general convergence result.
\newblock {\em Methods Appl. Anal.}, 16(3):321--340, 2009.

\bibitem{BP}
G.~Barles and B.~Perthame.
\newblock Dirac concentrations in {L}otka-{V}olterra parabolic {PDE}s.
\newblock {\em Indiana Univ. Math. J.}, 57(7):3275--3301, 2008.

\bibitem{BJS}
H.~Berestycki, T.~Jin, and L.~Silvestre.
\newblock Propagation in a non local reaction diffusion equation with spatial
  and genetic trait structure.
\newblock {\em Nonlinearity}, 29(4):1434--1466, 2016.

\bibitem{BouinCalvez}
E.~Bouin and V.~Calvez.
\newblock Travelling waves for the cane toads equation with bounded traits.
\newblock {\em Nonlinearity}, 27(9):2233--2253, 2014.

\bibitem{BHR2}
E.~Bouin, C.~Henderson, and L.~Ryzhik.
\newblock The {B}ramson logarithmic delay in the cane toads equations.
\newblock {\em Quart. Appl. Math.}, 75(4):599--634, 2017.

\bibitem{BHR3}
E.~Bouin, C.~Henderson, and L.~Ryzhik.
\newblock Super-linear spreading in local and non-local cane toads equations.
\newblock {\em J. Math. Pures Appl. (9)}, 108(5):724--750, 2017.

\bibitem{BouinMirrahimi}
E.~Bouin and S.~Mirrahimi.
\newblock A {H}amilton-{J}acobi approach for a model of population structured
  by space and trait.
\newblock {\em Commun. Math. Sci.}, 13(6):1431--1452, 2015.

\bibitem{brezis}
H.~Brezis.
\newblock {\em Functional Analysis, Sobolev Spaces and Partial Differential
  Equations}.
\newblock Universitext. Springer New York, 2010.

\bibitem{Harnack}
J.~Busca and B.~Sirakov.
\newblock Harnack type estimates for nonlinear elliptic systems and
  applications.
\newblock {\em Annales de l'Institut Henri Poincare (C) Non Linear Analysis},
  21(5):543 -- 590, 2004.

\bibitem{ChampagnatMeleard}
N.~Champagnat and S.~M\'{e}l\'{e}ard.
\newblock Invasion and adaptive evolution for individual-based spatially
  structured populations.
\newblock {\em J. Math. Biol.}, 55(2):147--188, 2007.

\bibitem{DJMP}
O.~Diekmann, P.-E. Jabin, S.~Mischler, and B.~Perthame.
\newblock The dynamics of adaptation: An illuminating example and a
  hamilton–jacobi approach.
\newblock {\em Theoretical Population Biology}, 67(4):257 -- 271, 2005.

\bibitem{MD.UD:03}
M.~Doebeli and U.~Dieckmann.
\newblock Speciation along environmental gradients.
\newblock {\em Nature}, 421:259--264, 01 2003.

\bibitem{Evans_visco_1}
L.~C. Evans.
\newblock The perturbed test function method for viscosity solutions of
  nonlinear {PDE}.
\newblock {\em Proc. Roy. Soc. Edinburgh Sect. A}, 111(3-4):359--375, 1989.

\bibitem{Evans-PDE}
L.~C. Evans.
\newblock {\em Partial differential equations}, volume~19 of {\em Graduate
  Studies in Mathematics}.
\newblock American Mathematical Society, Providence, RI, second edition, 2010.

\bibitem{Opt_geom_1}
L.~C. Evans and P.~E. Souganidis.
\newblock A {PDE} approach to geometric optics for certain semilinear parabolic
  equations.
\newblock {\em Indiana Univ. Math. J.}, 38(1):141--172, 1989.

\bibitem{Opt_geom_2}
M.~Freidlin.
\newblock Limit theorems for large deviations and reaction-diffusion equations.
\newblock {\em Ann. Probab.}, 13(3):639--675, 1985.

\bibitem{Trudinger}
D.~Gilbarg and N.S. Trudinger.
\newblock {\em Elliptic Partial Differential Equations of Second Order}.
\newblock Classics in Mathematics. Springer Berlin Heidelberg, 2001.

\bibitem{RH.JD.TH:12}
R.~Hermsen, J.~B. Deris, and T.~Hwa.
\newblock On the rapidity of antibiotic resistance evolution facilitated by a
  concentration gradient.
\newblock {\em Proc. Nat. Acad. Sci. USA}, 109:10775-- 10780, 2012.

\bibitem{HK.AL:2006}
H.~Kokko and A.~L{\'o}pez-Sepulcre.
\newblock From individual dispersal to species ranges: Perspectives for a
  changing world.
\newblock 313(5788):789--791, 2006.

\bibitem{ecology1}
A.~Kremer, O.~Ronce, J.~J. Robledo-Arnuncio, F.~Guillaume, G.~Bohrer,
  R.~Nathan, J.~R. Bridle, R.~Gomulkiewicz, E.~K. Klein, K.~Ritland,
  A.~Kuparinen, S.~Gerber, and S.~Schueler.
\newblock Long-distance gene flow and adaptation of forest trees to rapid
  climate change.
\newblock {\em Ecology Letters}, 15(4):378--392, 2012.

\bibitem{LamLou}
K.-Y. Lam and Y.~Lou.
\newblock An integro-{PDE} model for evolution of random dispersal.
\newblock {\em J. Funct. Anal.}, 272(5):1755--1790, 2017.

\bibitem{papier2}
A.~L{\'e}culier, S.~Mirrahimi, and J.-M. Roquejoffre.
\newblock Propagation in a fractional reaction--diffusion equation in a
  periodically hostile environment.
\newblock {\em Journal of Dynamics and Differential Equations}, pages 1--28,
  2020.

\bibitem{article3}
A.~L{\'e}culier and J.-M. Roquejoffre.
\newblock Properties of steady states for a class of non-local fisher kpp
  equations in general domains.
\newblock {\em preprint}, 2020.

\bibitem{Mirrahimi12}
S.~Mirrahimi.
\newblock Adaptation and migration of a population between patches.
\newblock {\em Discrete Contin. Dyn. Syst. Ser. B}, 18(3):753--768, 2013.

\bibitem{Mirrahimi17}
S.~Mirrahimi.
\newblock A hamilton-jacobi approach to characterize the evolutionary
  equilibria in heterogeneous environments.
\newblock {\em Mathematical Models and Methods in Applied Sciences},
  27(13):2425--2460, 2017.

\bibitem{MirrahimiGandon}
S.~Mirrahimi and S.~Gandon.
\newblock Evolution of specialization in heterogeneous environments:
  Equilibrium between selection, mutation and migration.
\newblock {\em Genetics}, 214(2):479--491, 2020.

\bibitem{PerthameSouganidis}
B.~Perthame and P.~E. Souganidis.
\newblock Rare mutations limit of a steady state dispersal evolution model.
\newblock {\em Math. Model. Nat. Phenom.}, 11(4):154--166, 2016.

\bibitem{JP.NB:05}
J.~Polechov\`a and N.~H. Barton.
\newblock Speciation through competition: a critical review.
\newblock {\em Evolution}, 59(6):1194--210, 2005.

\bibitem{Smoller}
J.~Smoller.
\newblock {\em Shock waves and reaction-diffusion equations}, volume 258 of
  {\em Grundlehren der Mathematischen Wissenschaften [Fundamental Principles of
  Mathematical Science]}.
\newblock Springer-Verlag, New York-Berlin, 1983.

\bibitem{Turanova}
O.~Turanova.
\newblock On a model of a population with variable motility.
\newblock {\em Math. Models Methods Appl. Sci.}, 25(10):1961--2014, 2015.

\bibitem{Whitlock2015}
M.~C. Whitlock.
\newblock Modern approaches to local adaptation.
\newblock {\em The American Naturalist}, 186(S1):S1--S4, 2015.
\newblock PMID: 26098334.

\end{thebibliography}

\end{document}